\documentclass[pdflatex,sn-mathphys,a4paper]{sn-jnl}% Math and Physical Sciences Reference Style
\usepackage{amsmath}
\usepackage{enumitem}
\usepackage{cleveref}

\jyear{2021}%

\theoremstyle{thmstyleone}%
\newtheorem{theorem}{Theorem}[section]% meant for sectionwise numbers
\newtheorem{proposition}[theorem]{Proposition}% 
\newtheorem{lem}[theorem]{Lemma}% 
\newtheorem{rem}[theorem]{Remark}%

\theoremstyle{thmstyletwo}%

\theoremstyle{thmstylethree}%
\newtheorem{definition}{Definition}%

\newcommand\D{\displaystyle}
\newcommand\dS{\textnormal{d}S}
\newcommand\dx{\textnormal{d}x}
\newcommand\dt{\textnormal{d}t}
\newcommand\dd{\textnormal{d}}
\newcommand{\cts}[1]{\left[#1\right]} 
\newcommand{\pts}[1]{\left(#1\right)} 
\newcommand{\abs}[1]{\left\lvert #1\right \rvert}   

\def\Xint#1{\mathchoice
{\XXint\displaystyle\textstyle{#1}}%
{\XXint\textstyle\scriptstyle{#1}}%
{\XXint\scriptstyle\scriptscriptstyle{#1}}%
{\XXint\scriptscriptstyle\scriptscriptstyle{#1}}%
\!\int}
\def\XXint#1#2#3{{\setbox0=\hbox{$#1{#2#3}{\int}$ }
\vcenter{\hbox{$#2#3$ }}\kern-.6\wd0}}

\def\dashint{\Xint-}

 \normalbaroutside
 
 \newcommand{\nor}[1]{\left\|#1\right\|}   
 \newcommand{\lvs}[1]{\left\{#1\right\}} 
 
\begin{document}

\title[Controllability of semilinear shadow systems]{Controllability of some semilinear shadow reaction-diffusion systems}


\author[]{\fnm{V\'ictor} \sur{Hern\'andez-Santamar\'ia}}\email{victor.santamaria@im.unam.mx}

\author[]{\fnm{Alberto} \sur{Pe\~{n}a-Garc\'ia}\email{pgbeto87@gmail.com}}

\affil[]{\orgdiv{Instituto de Matem\'{a}ticas}, \orgname{Universidad Nacional Aut\'{o}noma de M\'{e}xico}, \orgaddress{\street{Circuito Exterior C.U.,} \city{CDMX}, \postcode{04510}, \country{Mexico}}}


\abstract{ 
The shadow limit is a versatile tool used to study the reduction of reaction-diffusion systems into simpler PDE-ODE models by letting one of the diffusion coefficients tend to infinity. This reduction has been used to understand different qualitative properties and their interplay between the original model and its reduced version. The aim of this work is to extend previous results about the controllability of linear reaction-diffusion equations and how this property is inherited by the corresponding shadow model. Defining a suitable class of nonlinearities and improving some uniform Carleman estimates, we extend the results to the semilinear case and  prove that the original model is null-controllable and that the shadow limit preserves this important feature. 
}

\keywords{Shadow limit, semilinear reaction-diffusion equations, uniform null-controllability, Carleman estimates.}

\pacs[MSC Classification]{35K57, 93B05, 93B07, 93C20.}

\maketitle

\section{Introduction}\label{seccion-1}

\subsection{Motivation}

Reaction-diffusion systems have been established as one of the main tools for modeling biological, chemical, and biological processes. These systems are commonly composed by coupled semi-linear parabolic equations endowed with zero-flux boundary conditions. To fix ideas, let us consider the following model
\begin{equation}\label{eq:r-d_intro}
    \begin{cases}
    y_t-\sigma_1\Delta y = f(y,z) &\textnormal{in } (0,T)\times \Omega, \\
    z_t-\sigma_2\Delta z= g(y,z) &\textnormal{in } (0,T)\times \Omega, \\
    \D \frac{\partial y}{\partial \nu}=\frac{\partial z}{\partial \nu}=0 &\text{on } (0,T)\times\partial \Omega, \\
    y(0,\cdot)=y^0,\quad z(0,\cdot)=z^0  &\text{in }\ \Omega. 
    \end{cases}
\end{equation}

Here, $T>0$ is given, $\Omega\subset \mathbb{R}^{N}$ ($N\in\mathbb N^*$) is a nonempty, bounded set with smooth boundary $\partial \Omega$ and $\nu$ denotes the unit outward normal vector to $\partial\Omega$. In the remainder of this document, to abridge the notation, we set $Q_T:=(0,T)\times \Omega$ and $\Sigma_T=(0,T)\times \partial \Omega$.

In system \eqref{eq:r-d_intro}, $f$ and $g$ are two suitable nonlinear functions, $\sigma_1,\sigma_2>0$ represent the diffusion rates of the variables $y=y(t,x)$ and $z=z(t,x)$, respectively, and $y^0, z^0$ are their corresponding initial data. 

The paradigmatic model \eqref{eq:r-d_intro} has been studied extensively from many different angles: nonnegativity properties, monotonicity, entropy inequalities, long-time behavior, blow-up, pattern formation, among others. The literature is very rich so we refer the reader (for instance) to the monographs \cite{OL01}, \cite{Mur02}, and \cite{Per15} and the references therein for further material. 

As pointed out in \cite{Kow21}, reaction-difussion systems like \eqref{eq:r-d_intro} can exhibit quite complex structures and from a mathematical point of view it makes sense to reduce their complexity to understand the dynamics of the the full system while preserving its main properties. 

In this direction, a successful model reduction tool can be found in the seminal papers \cite{Kee78} and \cite{Nis82}. In such works, the idea is to study the limit of system \eqref{eq:r-d_intro} when one diffusion coefficient is fixed and the other tends to infinity. For instance, if $\sigma_1=1$ and $\sigma:=\sigma_2\to+\infty$ in \eqref{eq:r-d_intro}, due to the boundary conditions, the variable $z$ becomes spatially homogeneous, i.e. $z(t,x)=\xi(t)$, and it can be proved that system \eqref{eq:r-d_intro} becomes
\begin{equation}\label{eq:shadow_intro}
    \begin{cases}
    y_t-\Delta y = f(y,\xi) &\text{in } Q_T, \\
    \D \dot{\xi}=\dashint_{\Omega} g(y(t,x),\xi(t))\dx &\text{in } (0,T), \\
    \D \frac{\partial y}{\partial \nu}=0 &\text{on } \Sigma_T, \\
    \D y(0,\cdot)=y^0 \quad\text{in } \Omega, \quad \xi(0)=\dashint_{\Omega}z^0 \dx,
    \end{cases}
\end{equation}
where for any $u\in L^1(\Omega)$ we write 
\begin{equation*}
    \dashint_{\Omega} u \dx:=\frac{1}{|\Omega|}\int_{\Omega} u(x)\dx
\end{equation*}
and where $\abs{\Omega}$ denotes the measure of the set $\Omega$. 

The resulting equation \eqref{eq:shadow_intro} is the so-called shadow system and the limit behavior between \eqref{eq:r-d_intro} and \eqref{eq:shadow_intro} and some related properties have been studied in numerous works, see e.g. \cite{Tak86,HS89,Miy05,LN09,KT17,MCM17,MCHSKS18,LMCM21}. As a model reduction, system \eqref{eq:shadow_intro} has been successfully used to understand some properties of the original system \eqref{eq:r-d_intro} and vice versa. For instance in \cite{Tak86}, under suitable assumptions (mainly on $f$ and $g$), a stationary solution of the shadow problem allows to find a stationary solution of the original one, while in \cite{Miy05} it is shown that stability properties of those stationary states is preserved between the models. In the same spirit, in \cite{HS89}, it has been shown that compact attractors for the original reaction-diffusion system and its shadow limit are closely related. We conclude this part by mentioning that not always the dynamics and properties of the shadow system and the original one are related: for instance, it may happen that under suitable conditions the solutions of the shadow system blow-up in finite time while the original one has global solutions in time (see \cite{LN09}).

%Note further that \eqref{eq:shadow_intro} is non-local due to the presence of the integral term in the second equation. 

%%%%%%%%%%%%%%%%%%%%%%%%%%%%%%%%%%%%%%%%%%%%%%%%%%%%%%%%%%%%%%
          
\subsection{Problem formulation and main result}\label{seccion-2}

Motivated by the discussion above, the aim of this paper is to study some controllability properties of \eqref{eq:shadow_intro} by using the shadow-model reduction of the original system \eqref{eq:r-d_intro}.

To formulate this problem, let $\omega\subset \Omega$ be a nonempty, possibly small, open set and denote by $\mathbf 1_{\omega}$ its characteristic function. 
%Hereinafter, we write $q_T:=(0,T)\times \omega$. 
We begin by considering the system  

\begin{equation}\label{eq:r-d_controlled-NL}
    \begin{cases}
    y_t-\Delta y = f(y,z)+h\mathbf{1}_{\omega} &\textnormal{in } Q_T, \\
    z_t-\sigma\Delta z= g(y,z) &\textnormal{in } Q_T, \\
    \D \frac{\partial y}{\partial \nu}=\frac{\partial z}{\partial \nu}=0 &\text{on } \Sigma_T, \\
    y(0,\cdot)=y^0,\quad z(0,\cdot)=z^0  &\text{in }\ \Omega, 
    \end{cases}
\end{equation}
where $h=h(t,x)$ is an external control force. 

In the remainder of this document, we shall make the following instrumental assumptions on the nonlinear reaction terms: $f,g\in C^1(\mathbb R^2;\mathbb R)$ are such that
\begin{enumerate}[label=(\textbf{H\arabic*)}]
\item $f$ and $g$ are globally Lipschitz functions with constants $C_f$ y $C_g$, respectively. \label{H1}
\item $f(0,0)=0$ and $g(0,0)=0$. \label{H2}
\item  \label{H3}
\begin{enumerate}[label=(\roman*)] 
 \item There exists a constant ${a}_{21}^0>0$, such that  $\frac{\partial g}{\partial y}(0,0)\geq {a}_{21}^0$ or $-\frac{\partial g}{\partial y}(0,0)\geq {a}_{21}^0$, and
\item   $\frac{\partial g}{\partial  y}(\delta \bar y,\delta \bar z)$  does not change sign for all $\delta\in\cts{0,1}$ and $(\bar{y},\bar{z})\in\mathbb R^2$. 
\end{enumerate}
\end{enumerate}
Hypotheses \ref{H1} is related to the well-posedness of \eqref{eq:r-d_controlled-NL}. Indeed, under this assumption, for any $(y^0,z^0)\in [L^2(\Omega)]^2$ and $h\in L^2(0,T;L^2(\omega))$, system \eqref{eq:r-d_controlled-NL} has a unique solution in the class
\begin{equation*}%\label{eq:regularity}
    (y,z)\in L^2(0,T;[H^1(\Omega)]^2)\cup C([0,T];[L^2(\Omega)]^2).
\end{equation*}
On the other hand, \ref{H2} ensures that the state $(0,0)$ is an equilibrium point for \eqref{eq:r-d_controlled-NL} whenever $h\equiv 0$, while \ref{H3} will be useful later on during the implementation of some control techniques. 

\begin{rem}
In this remark, we provide a family of functions that satisfies assumptions \ref{H1}-\ref{H3}.  We consider a particular family of sigmoid functions given by
\begin{equation*}
    f^k(y,z)= \frac{y+z}{(1+\abs{y+z}^k)^{\frac{1}{k}}},
\end{equation*}
for a parameter $k>0$. After some straightforward computations, we can see that
\begin{equation}\label{Derivate-sigmoid}
\frac{\partial f^k}{\partial y}(y,z)=
    \frac{\partial f^k}{\partial z}(y,z)=\frac{1}{(1+\abs{y+z}^k)^{\frac{1}{k}+1}} \quad \text{for}\quad (y,z)\in \mathbb{R}^2,
\end{equation}
and $k>0$. Thus,
\begin{enumerate}
\item From \eqref{Derivate-sigmoid} we have that $0\leq\frac{\partial f^k}{\partial y}\leq 1$ and $0\leq\frac{\partial f^k}{\partial z}\leq 1$. This imply that $f^k$ is Lipschitz with constant $1$.
\item Clearly, $f^k(0,0)=0$.
\item Substituting $(y,z)=(0,0)$ in \eqref{Derivate-sigmoid} we have $\frac{\partial f^k}{\partial y}(0,0)=1$. On the other hand, we note that for $(\bar y,\bar z)\in \mathbb R^2$ fixed, $\frac{\partial f^k}{\partial  y}(\delta \bar y,\delta \bar z)$  does not change sign for all $\delta\in\cts{0,1}$.
\end{enumerate}
By properties $1$--$3$ we have that the family of function $f^k$ with $k>0$ satisfies the assumptions \ref{H1}--\ref{H3}.

This example shows that is possible to find functions that satisfy the conditions \ref{H1}--\ref{H3}. Another family of functions that satisfies these conditions is given by $g^k(y,z)=\arctan({k(y+z)})$ for $k>0$.
\end{rem}
One of the classical goals in control theory is to find a control $h$ such that a given system is steered to rest. In our case, this translates into the following. 
\begin{definition}\label{def:null_controllable}
System \eqref{eq:r-d_controlled-NL} is said to be null-controllable at time $T$ if for any $(y^0,z^0)\in [L^2(\Omega)]^2$, there exists a control $h\in L^2(0,T;L^2(\omega))$ such that the associated controlled states $(y,z)$ satisfy
\begin{equation}\label{eq:zero_state}
    y(T,\cdot)=z(T,\cdot)=0 \quad\text{in $\Omega$}.
\end{equation}
\end{definition}

Notice that in \eqref{eq:r-d_controlled-NL}, the action of the control $h$ enters directly onto the first equation of the system, but the second one only sees the action indirectly through the reaction term $g(y,z)$. In this direction, there is an extensive literature on the controllability of parabolic systems with less controls than equations, see e.g. the survey \cite{AKBGBdT11} for results up to 2011 and \cite{LB19} for a more recent one. 

In our particular case, it is well-known that for fixed $\sigma>0$ and assuming that \ref{H1}--\ref{H3} hold, system \eqref{eq:r-d_controlled-NL} is null-controllable\footnote{This fact seems to be well-known among the control-of-PDE community although the exact proof is difficult to find in the literature. The closest one is contained in the work \cite{AKBD06} but there hypothesis \ref{H1} is dropped and only a local-controllability result (i.e. for small initial data) is obtained. Here, we shall prove this claim (see \Cref{T1} below) by paying special attention to the dependency of $\sigma$ in the estimates. This introduces additional difficulties in the analysis but the main ideas are there for proving a (non-uniform) result for any $\sigma>0$.}, so it is natural to ask if this property is inherited by its shadow limit, namely, we wonder if there exists a control $h\in L^2(0,T;L^2(\omega))$ (by abuse of notation) such that the system
\begin{equation}\label{eq:shadow_controller}
    \begin{cases}
    y_t-\Delta y = f(y,\xi)+h\mathbf{1}_{\omega} &\text{in } Q_T, \\
    \D \dot{\xi}=\dashint_{\Omega} g(y(t,x),\xi(t))\dx &\text{in } (0,T), \\
    \D \frac{\partial y}{\partial \nu}=0 &\text{on } \Sigma_T, \\
    \D y(0,\cdot)=y^0 \quad\text{in } \Omega, \quad \xi(0)=\dashint_{\Omega}z^0 \dx.
    \end{cases}
\end{equation}
verifies
\begin{equation}\label{eq:contr_shadow}
    y(T,\cdot)=0 \; \textnormal{in $\Omega$} \quad\text{and}\quad \xi(T)=0.
\end{equation}

%At this point, there are several comments to make.
In the linear case, i.e., $f(y,z)=ay+bz$ and $g(y,z)=cy+dz$ with $a,b,c,d\in\mathbb R$ and $c\neq 0$, the controllability of \eqref{eq:shadow_controller} in the sense \eqref{eq:contr_shadow} has been proved in \cite[Theorem 2]{Zuazua-Victor}. In that work, the controllability of the linear shadow model was obtained by studying the observability of the corresponding adjoint system, namely
\begin{equation}\label{eq:adj_shadow}
    \begin{cases}
    -\varphi_t-\Delta \varphi = a\varphi+c\,\theta &\text{in } Q_T, \\
    \D -\dot{\theta}=b\,\dashint_{\Omega} \varphi(t,x)\dx+d\,\theta &\text{in } (0,T), \\
    \D \frac{\partial \varphi}{\partial \nu}=0 &\text{on } \Sigma_T, \\
    \D \varphi(T,\cdot)=\varphi^T \quad\text{in } \Omega, \quad \theta(T)=\theta^T,
    \end{cases}
\end{equation}
where $\varphi^T\in L^2(\Omega)$ and $\theta^T\in\mathbb R$.

The main ideas to prove such result are to use Carleman estimates for parabolic equations for the first equation of \eqref{eq:adj_shadow} and define $\zeta(t):=\dashint_{\Omega}\varphi(t,x)\dx$ for obtaining the following ode-ode system (thanks to the homogeneous Neumann boundary condition)
\begin{equation}\label{eq:red_red_shadow}
    \begin{cases}
        -\dot{\zeta}=a\zeta+c\theta &\textnormal{in }(0,T), \\
        -\dot{\theta}=b\zeta+d\theta &\textnormal{in }(0,T),
    \end{cases}
\end{equation}
for which Kalman rank condition holds as long as $c\neq 0$. All in all, this showed in the linear case the suitability of the shadow model reduction in the context of controllability. So a natural open question is to verify that such reduction is still meaningful in the semilinear case.

 In this direction, the main contribution of this paper is the following. 
\begin{theorem}\label{thm:main_1}
Assume that hypotheses \ref{H1}--\ref{H3} hold. There exists a control $h\in L^2(0,T;L^2(\omega))$ such that the corresponding solution $(y,\xi)$ to \eqref{eq:shadow_controller} satisfies \eqref{eq:contr_shadow}.
\end{theorem}

As usual in other semilinear control problems, the methodology roughly consists in obtaining a linearization of the considered system, then studying the observability of the corresponding adjoint equation and finally use a fixed point method. However, as it has been pointed out in \cite[Section 6.1]{Zuazua-Victor}, linearizing direclty system \eqref{eq:shadow_controller} yields
\begin{equation}\label{eq:linearized_shadow}
    \begin{cases}
    y_t-\Delta y = a(t,x)y + b(t) \xi +h\mathbf{1}_{\omega} &\text{in } Q_T, \\
    \D \dot{\xi}=\dashint_{\Omega} c(t,x) y (t,x)\dx + d(t) \xi &\text{in } (0,T), \\
    \D \frac{\partial y}{\partial \nu}=0 &\text{on } \Sigma_T, \\
    \D y(0,\cdot)=y^0 \quad\text{in } \Omega, \quad \xi(0)=\dashint_{\Omega}z^0 \dx,
    \end{cases}
\end{equation}
for some coefficients $a,b,c,d$, and the dependency of $c$ on $x$ invalidates many ideas used for the linear constant-coefficient case. In particular, the shadow system \eqref{eq:linearized_shadow} cannot longer be reduced to \eqref{eq:red_red_shadow} which is crucial in the proof shown in \cite{Zuazua-Victor}.

To circumvent this, we shall use instead the approach of \cite{HSLB21} which uses as a starting point the controllability properties of the original system and then the convergence towards the shadow system. In more detail, the first step consists in proving the following uniform result. 
\begin{theorem}\label{T1}
Let $\sigma \geq 1$ and assume that \ref{H1}--\ref{H3} hold. There exists a control $h=h(\sigma)\in L^2(0,T;L^2(\omega))$ such that the corresponding solution $(y,z)$ to \eqref{eq:r-d_controlled-NL} satisfies \eqref{eq:zero_state}. Moreover, we have the following uniform bound on the control 
\begin{equation}\label{eq:unif_bound}
    \|h(\sigma)\|_{L^2(0,T;L^2(\omega))}\leq C \left(\|y^0\|_{L^2(\Omega)}+\|z^0\|_{L^2(\Omega)}\right),
\end{equation}
where $C>0$ is a constant independent of $\sigma$, $y^0$ and $z^0$. Moreover, the solution $(y,z)$ satisfies the uniform estimate
\begin{equation}\label{energy-estimate-sigma}
\begin{aligned}	&\nor{y}_{L^2\pts{0,T;H^1(\Omega)}}^2+\nor{z}_{L^2\pts{0,T;H^1(\Omega)}}^2\\
	&+\nor{y_t}_{L^2(0,T;(H^1(\Omega))')}^2+\nor{z_t}_{L^2(0,T;(H^1(\Omega))')}^2+ \sigma\iint_{Q_T}\abs{\nabla z}^2\,dx\,dt\\
	&\hspace{3cm}\leq C\pts{\nor{y^0}_{L^2(\Omega)}^2+\nor{z^0}_{L^2(\Omega)}^2}.
\end{aligned}
\end{equation}
for a constant $C>0$ depending at most on $\Omega$, $\omega$, $T$, $C_f$, and $C_g$.
\end{theorem} 
At this point, we will follow a very conventional route, that is, we linearize system \eqref{eq:r-d_controlled-NL} (see eq. \eqref{eq:r-d_controlled} below) and prove a uniform  controllability result with respect to $\sigma$. Then, by means of a fixed point argument, we will obtain a uniform result for the nonlinear system. Here, it is important to mention that the Carleman estimate used in \cite{Zuazua-Victor} to address the linear case cannot be directly used and some improvements had to be implemented to obtain \Cref{T1} (see the discussion below \Cref{Desigualdad-obs}).

With \Cref{T1} at hand, the second step is to build a sequence of controls $\big(h(\sigma)\big)_{\sigma\geq 1}$ for system \eqref{eq:r-d_controlled-NL} and use the uniform bound \eqref{eq:unif_bound} to extract a subsequence such that
\begin{equation}\label{eq:conv_control}
    h(\sigma)\rightharpoonup h \quad\textnormal{weakly in $L^2(0,T;L^2(\omega))$ \ as \ $\sigma\to +\infty$},
\end{equation}
for some $h\in L^2(0,T;L^2(\omega))$. Then, we will use the uniform estimates \eqref{energy-estimate-sigma} and the convergence \eqref{eq:conv_control} to implement a shadow limit to reduce the controlled system \eqref{eq:r-d_controlled-NL} into \eqref{eq:shadow_controller} while preserving the null-controllability property. For that, we will make an adaptation of the shadow model reduction techniques shown in \cite{HR00} and \cite{MCHSKS18}, each of them relying on different tools \textemdash compactness and variational results for the first and semigroup theory for the second\textemdash but complementing each other allowing us to prove \Cref{thm:main_1}. 

\subsection{Organization of the paper}

The rest of the paper is organized as follows. In \Cref{seccion-3} we will focus on obtaining a uniform controllability result (w.r.t $\sigma$) for a linearized version of system \eqref{eq:r-d_controlled-NL}. This result will be used later in \Cref{seccion-5} to perform a fixed point method yielding, in particular, the proof of \Cref{T1}. Lastly, \Cref{seccion-6} is devoted to implementing the shadow limit and obtaining the proof of \Cref{thm:main_1}.

\section{Null controllability of a linear coupled system}\label{seccion-3}

\subsection{An observability inequality}
In this part of the paper, we study the null controllability of the linear reaction-diffusion system with Neumann boundary conditions
\begin{equation}\label{eq:r-d_controlled}
    \begin{cases}
    y_t-\Delta y = a_{11} y + a_{12} z +h\mathbf{1}_{\omega} &\textnormal{in } Q_T, \\
    z_t-\sigma\Delta z= a_{21} y + a_{22} z &\textnormal{in } Q_T, \\
    \D \frac{\partial y}{\partial \nu}=\frac{\partial z}{\partial \nu}=0 &\text{on } \Sigma_T \\
    y(0,\cdot)=y^0,\quad z(0,\cdot)=z^0  &\text{in }\ \Omega, 
    \end{cases}
\end{equation}
where $a_{ij}\in L^\infty(Q_T)$, $i,j=1,2$. This will be the first step towards the proof of \Cref{T1}.

First, we prove the following observability inequality for the adjoint system, given by 
\begin{equation}\label{eq:adjoint}
\begin{cases}
    -\phi_t-\Delta\phi=a_{11}\phi+a_{21}\psi,&\textnormal{in } Q_T, \\
    -\psi_t-\sigma\Delta\psi=a_{12}\phi+a_{22}\psi,&\textnormal{in } Q_T, \\
    \dfrac{\partial \phi}{\partial \nu}=\dfrac{\partial \psi}{\partial \nu}=0, &\textnormal{on } \Sigma_T\\
    \phi(T,\cdot)=\phi_T,\quad \psi(T,\cdot)=\psi_T &\textnormal{in } \Omega. 
\end{cases}
\end{equation}
The result reads as follows. 
\begin{proposition}\label{Desigualdad-obs} 
Let $\sigma\geq 1$ and assume that there is a constant $\hat{a}_{21}>0$ such that $a_{21}(t,x)>\hat{a}_{21}$ or $-a_{21}(t,x)>\hat{a}_{21}$ for all $(t,x)\in (0,T)\times\omega$. Then, for all $(\phi_T,\psi_T)\in [L^2(\Omega)]^2$, the solutions $(\phi,\psi)$ of the adjoint system \eqref{eq:adjoint} verify
\begin{equation}\label{e1.44}  
	\int_{\Omega}\pts{\abs{\phi(0,x)}^2+\abs{\psi(0,x)}^2}\dx\leq e^{CK}\iint\limits_{(0,T)\times\omega}\abs{\phi}^2\dx\dt,
\end{equation}
where
\begin{equation}\label{e1.45}
	K=1+T^{-1}+T\sum_{i,j=1}^2\nor{a_{ij}}_{L^{\infty}}+\max_{i,j=1,2}\nor{a_{ij}}_{L^{\infty}}^{\frac{2}{3}},
\end{equation}
and $C>0$ is a constant independent of $\sigma$ and $T$
\end{proposition}
We postpone the proof of this result to the end of the section. For its proof, we improve a Carleman estimate that appears in \cite[Theorem 16]{Zuazua-Victor} in two ways: first, we allow that the coefficients $a_{ij}$, $i,j=1,2$, depend on the variables $t$ and $x$; secondly, we keep track of the dependency of the norms of $a_{ij}$ allowing us to obtain an explicit expression of the constant $K$. We shall note that this was not done even for the constant coefficient case in \cite{Zuazua-Victor}.

To prove \Cref{Desigualdad-obs}, we begin by recalling some instrumental definitions and results about Carleman estimates for heat equations with homogeneous Neumann boundary conditions. 

\begin{lem}[{\cite[Lemma 1.1]{FI96}}]
Let $\mathcal B\subset\subset \Omega$ be a non-empty open subset. Then, there exists $\eta^0\in C^2(\bar{\Omega})$ such that $\eta^0>0$ in $\Omega$, $\eta_0=0$ on $\partial\Omega$, and $\abs{\nabla\eta^0}>0$ in $\overline{\Omega\setminus\mathcal B}$.
\end{lem}

 For a parameter $\lambda>0$, we define 
\begin{equation}\label{e1.2}
\alpha(t,x):=\frac{e^{2\lambda\nor{\eta^0}_{\infty}}-e^{\lambda\eta^0(x)}}{t(T-t)}, \quad \xi(t,x):=\frac{e^{\lambda\eta^0(x)}}{t(T-t)}, 
\end{equation}  
\begin{equation}\label{e1.3}
\begin{aligned}
	\hat\alpha(t)&:=\max_{x\in\bar{\Omega}}\alpha(t,x), & \xi^*(t):=\max_{x\in \bar{\Omega}}\xi(t,x),\\
	\alpha^*(t)&:=\min_{x\in\bar{\Omega}}\alpha(t,x), & \hat\xi(t):=\min_{x\in \bar{\Omega}}\xi(t,x).
\end{aligned}
\end{equation}

The Carleman estimate for heat equations with homogeneous boundary conditions we will use reads as follows. 

\begin{lem}\label{Lema-1}
There exists $C=C(\Omega,\omega)$ and $\lambda_0=\lambda_0(\Omega,\omega)$ such that, for every $\lambda\geq\lambda_0$ there exists $s_0=s_0(\Omega,\omega, \lambda)$ such that the solution $q$ to
\begin{equation*}%\label{e1.5}
\begin{cases}
	-\hat\sigma q_t- \Delta q=g(t,x), & \text{in}  \quad Q_T,\\
	\D \frac{\partial q}{\partial \nu}=0, & \text{on}  \quad  \Sigma_T,\\
 	q(x,T)=q_T(x), & \text{in} \quad \Omega,
\end{cases}
\end{equation*}
satisfies for any $s\geq s_0(T+T^2)$, $q_T\in L^2(\Omega)$, and $g\in L^2(Q_T)$, the estimate
\begin{equation}\label{e.12}
	I(s,\hat\sigma;q)\leq C\pts{\iint_{Q_T} e^{-2s\alpha}\abs{g}^2 \dx\dt+s^3\iint\limits_{(0,T)\times\mathcal B} e^{-2s\alpha}\xi^3\abs{q}^2\dx\dt},
\end{equation}
where 
\begin{align*}
I(s,\hat{\sigma};q)&:=s^{-1}\iint_{Q_T}e^{-2s\alpha}\xi^{-1}\left(\hat{\sigma}^2|q_t|^2+\sum_{i,j=1}^{N}\left|\frac{\partial^2 q}{\partial x_{i}x_{j}}\right|^2\right)\dx\dt \\
&\quad +s\iint_{Q_T}e^{-2s\alpha}\xi|\nabla q|^2\dx\dt+s^3\iint_{Q_T}e^{-2s\alpha}\xi^3|q|^2\dx\dt.
\end{align*}
\end{lem}

Now, we are in position to present one of the main results of this section. 
\begin{proposition}\label{p1.2}
Under the conditions of the Proposition \ref{Desigualdad-obs}, there is a constant $C>0$ depending at most on $\Omega$, $\omega$ and $\hat{a}_{21}$, such that for any $(\phi_T,\psi_T)\in [L^2(\Omega)]^2$, the solutions $(\phi,\psi)$ of the adjoint \eqref{eq:adjoint} verify
\begin{equation*}%\label{e1.8}  
\begin{aligned}
	s^3\iint_{Q_T} e^{-2s\alpha}\xi^3\abs{\phi}^2\dx\dt&+s^3\iint_{Q_T} e^{-2s\alpha}\xi^3\abs{\psi}^2\dx\dt \\ &\leq Cs^8\iint_{(0,T)\times\omega}e^{-4s\alpha^*+2s\hat\alpha}\pts{\xi^*}^8\abs{\phi}^2\dx\dt,
\end{aligned}
\end{equation*}
for any 
\begin{equation*}
	s\geq C\pts{T+T^2+T^2\max_{i,j=1,2}\nor{a_{ij}}_{L^\infty}^{\frac{2}{3}}}.
\end{equation*}
\end{proposition}
\begin{proof}
Let us consider sets $\omega_i\subset \Omega$, $i=0,1$ such that 
\begin{equation}\label{e1.4}
	\omega_0\subset\subset\omega_1\subset\subset\omega.
\end{equation}
All along the proof, $C$ will be a generic positive constant that depends at most on $\Omega$, $\omega$ and $\hat{a}_{21}$ and that may change from line to line. For readability, we have divided the proof in four steps.

\textbf{Step 1: initial estimates.} We fix $\lambda$ to a value large enough and apply estimate \eqref{e.12} to the first equation of \eqref{eq:adjoint} with $\hat\sigma=1$, $\mathcal B=\omega_0$, and $g(t,x)=a_{11}(t,x)\phi+a_{21}(t,x)\psi$ to get
\begin{equation}\label{e1.9}
\begin{aligned}
	I(s,1;\phi) 
	\leq C\iint_{Q_T} e^{-2s\alpha}\abs{a_{11}(t,x)\phi+a_{21}(t,x)\psi}^2\dx\dt \\
	+Cs^3\iint\limits_{(0,T)\times\omega_0} e^{-2s\alpha}\xi^3\abs{\phi}^2\dx\dt.
\end{aligned}
\end{equation}
Similarly, we divide over $\sigma$ the second equation of \eqref{eq:adjoint} and once again apply estimate  \eqref{e.12} with $\hat\sigma=\sigma^{-1}$, $\mathcal B=\omega_0$, and $g(t,x)=\sigma^{-1}\pts{a_{12}(t,x)\phi+a_{22}(t,x)\psi}$
\begin{equation}\label{e1.10}
\begin{aligned}
	I(s,\sigma^{-1};\psi) 
	\leq C\iint_{Q_T} e^{-2s\alpha}\abs{\sigma^{-1}\pts{a_{12}(t,x)\phi+a_{22}(t,x)\psi}}^2\dx\dt \\
	+Cs^3\iint\limits_{(0,T)\times\omega_0} e^{-2s\alpha}\xi^3\abs{\psi}^2\dx\dt.
\end{aligned}
\end{equation}
Now, adding the inequalities \eqref{e1.9} and \eqref{e1.10}, using that $\sigma\geq 1$, the right-hand side can be estimated as follows
\begin{equation*}%\label{e1.11}
\begin{aligned}
	I(s,1;\phi)&+I(s,\sigma^{-1};\psi) \\
	\leq& C\iint_{Q_T} e^{-2s\alpha}\pts{\nor{a_{11}}_{L^\infty}^2 +\nor{a_{12}}_{L^\infty}^2}\abs{\phi}^2\dx\dt\\
	& \hspace{0.67cm}+C\iint_{Q_T} e^{-2s\alpha}\pts{\nor{a_{21}}_{L^\infty}^2 +\nor{a_{22}}_{L^\infty}^2}\abs{\psi}^2\dx\dt\\
	&\hspace{1.95cm}+Cs^3\iint\limits_{(0,T)\times\omega}e^{-2s\alpha}\xi^3\pts{\abs{\phi}^2+\abs{\psi}^2}\dx\dt. 
\end{aligned}
\end{equation*}
Note that $\xi^{-1}(t,x)\leq\frac{T^2}{4}$. This implies  
\begin{equation}\label{e1.12}
\begin{aligned}
	I(s,1;\phi)+&I(s,\sigma^{-1};\psi)\\
	&\leq C\pts{T^2\max_{i,j=1,2}\nor{a_{ij}}_{L^\infty}^{\frac{2}{3}}}^3\iint_{Q_T} e^{-2s\alpha}\xi^3\pts{\abs{\phi}^2+\abs{\psi}^2}\dx\dt \\
	&\hspace{1.7cm}+Cs^3\iint\limits_{(0,T)\times\omega}e^{-2s\alpha}\xi^3\pts{\abs{\phi}^2+\abs{\psi}^2}\dx\dt. 
\end{aligned}
\end{equation} 
Under the assumptions of Lemma~\ref{Lema-1}, we have that $s\geq C\pts{T+T^2}$ for some $C>0$ depending on $\Omega$ and $\omega$. If we choose 
\begin{equation}\label{e1.13}
	s\geq C\pts{T+T^2+T^2\max_{i,j=1,2}\nor{a_{ij}}_{L^\infty}^{\frac{2}{3}}},
\end{equation}
for a (possibly larger) constant $C>0$ independent of $a_{ij}$ and $T$, we can absorb the first term in the right-hand side of \eqref{e1.12}, that is
\begin{equation}\label{e1.14}
	I(s,1;\phi)+I(s,\sigma^{-1};\psi) 	
	\leq Cs^3\iint\limits_{(0,T)\times\omega_0}e^{-2s\alpha}\xi^3\pts{\abs{\phi}^2+\abs{\psi}^2}\dx\dt. 
\end{equation}

\textbf{Step 2: local energy estimates.} In this step we estimate the local integral term of $\psi$ in the right-hand side of \eqref{e1.14}. Let $\eta\in C_0^{\infty}(\omega_1)$ be such that
\begin{equation}\label{e1.15}
\begin{aligned}
	&0\leq\eta\leq 1, \quad\text{ in } \omega_1,  \\
	&\eta(x)=1, \quad \text{for all}\quad x\in\omega_0, 
\end{aligned}
\end{equation}
and assume (without loss of generality) that $-a_{21}\geq \hat{a}_{21}>0$ in $(0,T)\times \omega$. Using the first equation of \eqref{eq:adjoint}, we obtain
\begin{equation}\label{e1.16}
\begin{aligned}
	s^3\iint\limits_{(0,T)\times \omega_0}&e^{-2s\alpha}\xi^3\abs{\psi}^2\dx\dt \leq s^3\iint\limits_{(0,T)\times\omega_1} \eta(x)e^{-2s\alpha}\xi^3\abs{\psi}^2\dx\dt \\
	&=-s^3\iint\limits_{(0,T)\times\omega_1}\eta(x)e^{-2s\alpha}\xi^3\frac{\psi}{a_{21}(t,x)}\pts{\phi_t+\Delta\phi+a_{11}(t,x)\phi}\dx\dt \\ 
	%&\leq \frac{s^3}{\hat a_{21}}\iint\limits_{(0,T)\times\Omega}\eta(x)e^{-2s\alpha}\xi^3\psi\pts{\phi_t+\Delta\phi+a_{11}(t,x)\phi}\,dx\,dt\\
	&\leq \abs{\frac{s^3}{\hat a_{21}}\iint_{Q_T}\eta(x)e^{-2s\alpha}\xi^3\psi\pts{\phi_t+\Delta\phi+a_{11}(t,x)\phi}\dx\dt}.
\end{aligned}
\end{equation}

Integrating by parts in time and using triangle inequality yields  
\begin{equation*}%\label{e1.17}
\begin{aligned} 
	&s^3\iint\limits_{(0,T)\times\omega_0}e^{-2s\alpha}\xi^3\abs{\psi}^2\,dx\,dt\\
	&\leq \abs{\frac{s^3}{\hat a_{21}}\iint\limits_{(0,T)\times\Omega}\pts{\eta(x) \pts{e^{-2s\alpha}\xi^3}_t\psi \phi+\eta(x)e^{-2s\alpha}\xi^3\psi_t \phi}\,dx\,dt} \\
	&\hspace{0.2cm}+\abs{\frac{s^3}{\hat{a}_{21}}\iint\limits_{(0,T)\times \Omega}\eta(x)e^{-2s\alpha}\xi^3\psi\Delta\phi\,dx\,dt} +\abs{\frac{s^3}{\hat{a}_{21}}\iint\limits_{(0,T)\times \Omega}a_{11}(t,x)\eta(x)e^{-2s\alpha}\xi^3\psi\phi\,dx\,dt},	
\end{aligned}
\end{equation*}
whence, substituting the second equation of \eqref{eq:adjoint} in the first term of the right-hand side of the previous inequality we get
\begin{equation}\label{e1.18}
\small
\begin{aligned}
	&s^3\iint\limits_{(0,T)\times\omega_1}e^{-2s\alpha}\xi^3\abs{\psi}^2\dx\dt 
	\leq \abs{\frac{s^3}{\hat a_{21}}\iint_{Q_T}\eta(x) \pts{e^{-2s\alpha}\xi^3}_t\psi\phi\,\dx\dt }\\
	 &+\abs{\frac{s^3}{\hat a_{21}}\iint_{Q_T}\eta(x)e^{-2s\alpha}\xi^3\sigma\Delta\psi\phi\,\dx\dt} 
	 +\abs{\frac{s^3}{\hat a_{21}}\iint_{Q_T}\eta(x)e^{-2s\alpha}\xi^3a_{12}(t,x)\abs{\phi}^2\dx\dt}\\
	 &\hspace{1cm}+\abs{\frac{s^3}{\hat a_{21}}\iint_{Q_T}\eta(x)e^{-2s\alpha}\xi^3a_{22}(t,x)\psi\phi\, \dx\dt} \\
	&\hspace{2cm}+\abs{\frac{s^3}{\hat{a}_{21}}\iint_{Q_T}\eta(x)e^{-2s\alpha}\xi^3\psi\Delta\phi\,\dx\dt} \\
	&\hspace{3cm}+\abs{\frac{s^3}{\hat{a}_{21}}\iint_{Q_T} a_{11}(t,x)\eta(x)e^{-2s\alpha}\xi^3\psi\phi \,\dx \dt}\\
	& \hspace{4cm}=\sum_{i=1}^6\abs{K_i}.			
\end{aligned}
\end{equation}

Before estimating the terms $K_i$ for $i\in\{1,\ldots,6\}$, we establish the following estimates for the derivatives of the Carleman weights
\begin{align}\label{eq:deriv_t_weight}
 \abs{\pts{e^{-2s\alpha}\xi^3}_t}&\leq Cs^2\xi^5 e^{-2s\alpha}, \quad 
 \abs{\pts{e^{-2s\hat\alpha}{\hat\xi}^{-2}}_t}\leq Cs^2e^{-2s\hat\alpha},\\ \label{eq:deriv_x_weight}
 \abs{\nabla\pts{e^{-2s\alpha}\xi^3}}&\leq Ce^{-2s\alpha}s\xi^4, \quad 
 \abs{\Delta\pts{e^{-2s\alpha}\xi^3}}\leq Ce^{-2s\alpha}s^2\xi^5.
\end{align}
From \eqref{eq:deriv_t_weight}, together with  Cauchy–Schwarz and Young inequalities, we have
\begin{equation}\label{e1.21}
\begin{aligned}
	&\abs{K_1}\leq 	C\delta s^3\iint_{Q_T}e^{-2s\alpha}\xi^3\abs{\psi}^2\dx\dt \\
	&\hspace{2.5cm}+C_{\delta}s^7\iint\limits_{(0,T)\times\omega_1}e^{-2s\alpha}\xi^7\abs{\phi}^2\dx\dt,
\end{aligned}
\end{equation}
for any $\delta\in(0,1)$. 

By definition of the functions $\hat\alpha$, $\alpha^*$, $\hat\xi$ y $\xi^*$ given in \eqref{e1.3} and applying Cauchy–Schwarz and Young inequalities 
\begin{equation}\label{e1.22}
\begin{aligned}
	&\abs{K_2}\leq C\sigma^2s^{-2}\iint_{Q_T}e^{-2s\hat\alpha}\hat\xi^{-2}\abs{\Delta \psi}^2\dx\dt\\
	&\hspace{2.5cm}+Cs^8\iint\limits_{(0,T)\times\Omega}\eta^2(x)e^{-4s\alpha+2s\hat\alpha}\hat\xi^{2}\xi^6\abs{\phi}^2\dx\dt\\
	&\qquad \leq  C\sigma^2s^{-2}\iint_{Q_T}e^{-2s\hat\alpha}\hat\xi^{-2}\abs{\Delta \psi}^2\dx\dt\\
	&\hspace{2.9cm}+Cs^8\iint\limits_{(0,T)\times\omega_1}e^{-4s\alpha^*+2s\hat\alpha}\pts{\xi^*}^8\abs{\phi}^2\dx\dt.
\end{aligned}
\end{equation}

To estimate $K_3$, we recall that \eqref{e1.13} holds. This implies that
\begin{equation}\label{e1.23}
	\nor{a_{ij}}_{L^{\infty}}\leq \pts{\frac{s}{CT^2}}^{\frac{3}{2}}\leq C s^{\frac{3}{2}}\xi^{\frac{3}{2}},
\end{equation}
for $i,j=1,2$. Using the previous estimate for the coefficient $a_{12}$ we get
\begin{equation}\label{e1.24}
\begin{aligned}
	&\abs{K_3}\leq Cs^{\frac{9}{2}}\iint\limits_{(0,T)\times\omega_1}e^{-2s\alpha}\xi^{\frac{9}{2}}\abs{\phi}^2\dx\dt.
\end{aligned}
\end{equation}
In the same spirit, using \eqref{e1.23} and applying the Cauchy-Schwarz and Young inequalities, we can estimate
\begin{equation}\label{e1.25}
\begin{aligned}
	\abs{K_i}&\leq C\iint\limits_{(0,T)\times\Omega}\eta(x)e^{-2s\alpha}s^{\frac{9}{2}}\xi^{\frac{9}{2}}\abs{\psi}\abs{\phi} \dx\dt\\
	&\leq C\delta s^3\iint_{Q_T} e^{-2s\alpha}\xi^3\abs{\psi}^2\dx\dt  +C_{\delta}s^6\iint\limits_{(0,T)\times\omega_1} e^{-2s\alpha}\xi^{6}\abs{\phi}^2\dx\dt,
\end{aligned}
\end{equation}
for $i=4,6$ and any $\delta\in(0,1)$. 

Now, we estimate the term $\abs{K_5}$. Firstly, we integrate by parts the term $K_5$
\begin{equation*}%\label{e1.27}
%\footnotesize
\begin{aligned}
		&K_5=\frac{s^3}{\hat{a}_{21}}\pts{-\iint_{Q_T}\nabla\lvs{\eta(x)e^{-2s\alpha}\xi^3\psi}\nabla\phi\dx\dt+\iint_{\Sigma_T}\frac{\partial\phi}{\partial \nu}\eta(x)e^{-2s\alpha}\xi^3\psi\dS\dt}\\
		&=-\frac{s^3}{\hat{a}_{21}}\iint_{Q}\nabla\lvs{\eta(x)e^{-2s\alpha}\xi^3}\psi\nabla\phi\,\dx\dt-\frac{s^3}{\hat{a}_{21}}\iint_{Q}\eta(x)e^{-2s\alpha}\xi^3\nabla\psi\nabla\phi\,\dx\dt,
\end{aligned}
\end{equation*}
where we have used that $\phi$ has homogeneous Neumann boundary conditions. Once again, integrating by parts in $Q_T$ 
\begin{align*}\notag
		K_5&=\frac{s^3}{\hat{a}_{21}}\iint_{Q_T}\Delta\lvs{\eta(x)e^{-2s\alpha}\xi^3}\psi\phi\,\dx\dt+\frac{s^3}{\hat{a}_{21}}\iint_{Q_T}\nabla\lvs{\eta(x)e^{-2s\alpha}\xi^3}\nabla\psi\phi\,\dx\dt \\ \notag
		&\quad -\frac{s^3}{\hat{a}_{21}}\iint_{\Sigma} \phi\psi\frac{\partial\pts{\eta e^{-2s\alpha}\xi^3}}{\partial \nu} \dS \dt- \frac{s^3}{\hat{a}_{21}}\iint_{\Sigma} \phi \frac{\partial \psi}{\partial\nu}e^{-2s\alpha}\xi^3 \dS \dt \\  \notag
		&\quad +\frac{s^3}{\hat{a}_{21}}\iint_{Q_T}\nabla\lvs{\eta(x)e^{-2s\alpha}\xi^3}\nabla\psi\phi \,\dx \dt+\frac{s^3}{\hat{a}_{21}}\iint_{Q_T}\eta(x)e^{-2s\alpha}\xi^3\Delta\psi\phi \,\dx\dt	\\ \notag
		&=\frac{s^3}{\hat{a}_{21}}\iint_{Q_T}\Delta\lvs{\eta(x)e^{-2s\alpha}\xi^3}\psi\phi \,\dx\dt+\frac{2s^3}{\hat{a}_{21}}\iint_{Q_T}\nabla\lvs{\eta(x)e^{-2s\alpha}\xi^3}\nabla\psi\phi \,\dx \dt \\ %\label{e1.28}
		&\quad +\frac{s^3}{\hat{a}_{21}}\iint_{Q_T}\eta(x)e^{-2s\alpha}\xi^3\Delta\psi\phi \, \dx \dt,
\end{align*}
where using that $\eta$ is compactly supported in $\Omega$ (recall eq. \eqref{e1.15}) we have that  $\eta=\nabla\eta=0$ on $\partial\Omega$ and the two integrals  on the boundary vanish.

Finally, using \eqref{eq:deriv_x_weight} we obtain
\begin{equation*}
\begin{aligned}
		\abs{K_5}&\leq C\iint_{Q}\abs{\Delta\eta}e^{-2s\alpha}s^3\xi^3\abs{\psi}\abs{\phi}\dx\dt+C\iint_{Q}\abs{\nabla\eta}e^{-2s\alpha}s^4\xi^4\abs{\psi}\abs{\phi}\dx\dt\\
		&\quad+C\iint_{Q}\eta(x)Ce^{-2s\alpha}s^5\xi^5\abs{\psi}\abs{\phi}\dx\dt + C\iint_{Q}\abs{\nabla\eta}e^{-2s\alpha}s^3\xi^3\abs{\nabla\psi}\abs{\phi}\dx\dt\\
		&\quad+C\iint_{Q}\eta(x)Ce^{-2s\alpha}s^4\xi^4\abs{\nabla\psi}\abs{\phi}\dx\dt +C\iint_{Q}\eta(x)e^{-2s\alpha}s^3\xi^3\abs{\Delta\psi}\abs{\phi}\dx\dt.
\end{aligned}
\end{equation*}
By Young inequality and the fact that $\eta\in C_0^\infty(\omega_1)$ we get
\begin{equation}\label{e1.29}
\small
\begin{aligned}
	\abs{K_5}&\leq C_{\delta}s^3\iint\limits_{(0,T)\times\omega_1}e^{-2s\alpha}\xi^3\abs{\phi}^2\dx\dt+ 2C_{\delta}s^5\iint\limits_{(0,T)\times\omega_1}e^{-2s\alpha}\xi^5\abs{\phi}^2\dx\dt\\
& \quad+ 3C_{\delta}s^7\iint\limits_{(0,T)\times\omega_1}e^{-2s\alpha}\xi^7\abs{\phi}^2\dx\dt+3C\delta s^3\iint_{Q_T} e^{-2s\alpha}\xi^3\abs{\psi}^2\dx\dt\\
		&\quad +2C\delta s\iint_{Q_T}e^{-2s\alpha}\xi\abs{\nabla\psi}^2\dx\dt+C\delta s^{-1}\iint_{Q_T}e^{-2s\alpha}\xi^{-1}\abs{\Delta\psi}^2\dx\dt,
\end{aligned}
\end{equation}
for any $\delta\in(0,1)$.

Putting together \eqref{e1.16}, \eqref{e1.18} and the estimates for $\abs{K_i}$, for $i=1,\ldots, 6$, obtained in \eqref{e1.21}, \eqref{e1.22}, \eqref{e1.24}, \eqref{e1.25} and \eqref{e1.29} yields 
\begin{equation*}%\label{e1.30}
\begin{aligned}
	s^3&\iint\limits_{(0,T)\times\omega_0}e^{-2s\alpha}\xi^3\abs{\psi}^2\dx\dt\\
	&\leq  6C\delta s^3\iint_{Q_T}e^{-2s\alpha}\xi^3\abs{\psi}^2\dx\dt +4C_{\delta}s^7\iint\limits_{(0,T)\times\omega_1}e^{-2s\alpha}\xi^7\abs{\phi}^2\dx\dt  \\ %K_1
		  &\quad +C\sigma^2s^{-2}\iint_{Q_T}e^{-2s\hat\alpha}\hat\xi^{-2}\abs{\Delta \psi}^2\dx\dt	+Cs^8\iint\limits_{(0,T)\times\omega_1}e^{-4s\alpha^*+2s\hat\alpha}\pts{\xi^*}^8\abs{\phi}^2\dx\dt  \\ %K_2 
		  &\quad+Cs^{\frac{9}{2}}\iint\limits_{(0,T)\times\omega_1}e^{-2s\alpha}\xi^{\frac{9}{2}}\abs{\phi}^2\dx\dt+2C_{\delta}s^6\iint\limits_{(0,T)\times\omega_1} e^{-2s\alpha}\xi^{6}\abs{\phi}^2\dx\dt    \\  %K_3   %K_4
		  &\quad+C_{\delta}s^3\iint\limits_{(0,T)\times\omega_1}e^{-2s\alpha}\xi^3\abs{\phi}^2\dx\dt+ 2C_{\delta}s^5\iint\limits_{(0,T)\times\omega_1}e^{-2s\alpha}\xi^5\abs{\phi}^2\dx\dt\\
		&\quad+2C\delta s\iint_{Q_T}e^{-2s\alpha}\xi\abs{\nabla\psi}^2\dx\dt+C\delta s^{-1}\iint_{Q_T}e^{-2s\alpha}\xi^{-1}\abs{\Delta\psi}^2\dx\dt. %K_5  %K_6		
\end{aligned}
\end{equation*}
We obtain by the definition of $\hat\alpha$, $\alpha^*$ in \eqref{e1.3} and $\xi^{-1}\leq CT^2$ that  $s^p\xi^p e^{-2s\alpha}\leq Cs^8\xi^8e^{-4s\alpha^*+2s\hat\alpha}$ for any $p<8$. Then, using this in the localized terms, we get  
\begin{equation}\label{e1.31}
\begin{aligned}
	s^3\iint\limits_{(0,T)\times\omega_0}&e^{-2s\alpha}\xi^3\abs{\psi}^2\dx\dt\\
	\leq&  6C\delta s^3\iint_{Q_T}e^{-2s\alpha}\xi^3\abs{\psi}^2\dx\dt +2C\delta s\iint_{Q_T}e^{-2s\alpha}\xi\abs{\nabla\psi}^2\dx\dt \\
	&+C\delta s^{-1}\iint_{Q_T}e^{-2s\alpha}\xi^{-1}\abs{\Delta\psi}^2\dx\dt +C\sigma^2s^{-2}\iint\limits_{Q_T}e^{-2s\hat\alpha}\hat\xi^{-2}\abs{\Delta \psi}^2\dx\dt	\\
		  &+Cs^8\iint\limits_{(0,T)\times\omega_1}e^{-4s\alpha^*+2s\hat\alpha}\pts{\xi^*}^8\abs{\phi}^2\dx\dt, 
\end{aligned}
\end{equation}
for any $\delta\in(0,1)$.

\textbf{Step 3: a uniform global estimate.} Note that the first three terms in the right-hand side of \eqref{e1.31} are multiplied by a small parameter $\delta$ but the fourth one has the large parameter $\sigma$ in front of it. Thus, the idea is to estimate it uniformly with respect to $\sigma$. 

Using the second equation of \eqref{eq:adjoint}, multiplying by $e^{-2s\hat\alpha}{\hat{\xi}}^{-2}\Delta \psi$, integrating by parts, and applying the  Cauchy-Schwarz inequality we obtain
\begin{equation*}
\begin{aligned}
	&\sigma\iint_{Q_T}e^{-2s\hat\alpha}{\hat{\xi}}^{-2}\abs{\Delta\psi}^2\dx\dt\\
	&\leq\iint_{Q_T}e^{-2s\hat\alpha}{\hat{\xi}}^{-2}\nabla\psi\nabla\psi_t\dx\dt- \iint_{\Sigma}\frac{\partial \psi}{\partial \nu}\psi_te^{-2s\hat\alpha}{\hat{\xi}}^{-2}\dS\dt\\
	&\hspace{1cm}+\iint_{Q_T}e^{-2s\hat\alpha}{\hat{\xi}}^{-2}\nor{a_{12}}_{L^{\infty}}\abs{\Delta\psi}\abs{\phi}\dx\dt\\
	&\hspace{2cm}+\iint_{Q_T}e^{-2s\hat\alpha}{\hat{\xi}}^{-2}\nor{a_{22}}_{L^{\infty}}\abs{\Delta\psi}\abs{\psi}\dx\dt.
\end{aligned}
\end{equation*}  
Taking into account \eqref{e1.23} and the fact that $\hat{\xi}\leq \xi\leq C\hat{\xi}$ in $(0,T)\times\overline{\Omega}$,where $C$ is a constant only depending on $\Omega$ and $\omega$, we deduce 
\begin{equation}\label{e1.32}
\begin{aligned}
	&\sigma\iint_{Q_T}e^{-2s\hat\alpha}{\hat{\xi}}^{-2}\abs{\Delta\psi}^2\dx\dt\\
	&\leq \frac{1}{2}\iint_{Q_T}e^{-2s\hat\alpha}{\hat{\xi}}^{-2}\pts{\abs{\nabla\psi}^2}_t\dx\dt\\
	&\hspace{0.5cm}+2\delta\sigma\iint_{Q_T}e^{-2s\hat\alpha}{\hat{\xi}}^{-2}\abs{\Delta\psi}^2\dx\dt+C_{\delta}\sigma^{-1}s^3\iint_{Q_T}e^{-2s\hat\alpha}\hat\xi\abs{\phi}^2\,dx\,dt\\
	&\hspace{2cm}+C_{\delta}\sigma^{-1}s^3\iint_{Q_T}e^{-2s\hat\alpha}\hat\xi\abs{\psi}^2\dx\dt.
\end{aligned}
\end{equation}
In order to estimate the first integral at the right-hand side of inequality \eqref{e1.32}, we integrate on $(0,T)$ and continue by integrating by parts on $\Omega$, to obtain
\begin{equation}\label{e1.33}
\begin{aligned}
	&\frac{1}{2}\iint_{Q_T}e^{-2s\hat\alpha}{\hat{\xi}}^{-2}\pts{\abs{\nabla\psi}^2}_t\dx\dt\\
	&=\frac{1}{2}\pts{\left.\int_{\Omega}e^{-2s\hat\alpha}{\hat{\xi}}^{-2}\pts{\abs{\nabla\psi}^2}\dx \right|_0^T- \iint_{Q_T}\pts{e^{-2s\hat\alpha}{\hat\xi}^{-2}}_t\abs{\nabla\psi}^2\dx\dt}\\
	&= \frac{1}{2}\iint_{Q_T}\pts{e^{-2s\hat\alpha}{\hat\xi}^{-2}}_t\psi\Delta\psi\dx\dt \\
	&\leq Cs^2\iint_{Q_T}e^{-2s\hat\alpha}\abs{\psi}\abs{\Delta\psi}\dx\dt,
\end{aligned}
\end{equation}
where we have used that $\abs{\pts{e^{-2s\hat\alpha}{\hat\xi}^{-2}}_t}\leq Cs^2e^{-2s\hat\alpha}$ to get the last line of the above expression. Substituting \eqref{e1.33} into \eqref{e1.32} and applying Young's inequality yields
\begin{equation}\label{e1.34}
\begin{aligned}
	&\sigma\iint_{Q_T}e^{-2s\hat\alpha}{\hat{\xi}}^{-2}\abs{\Delta\psi}^2\dx\dt\\
	&\leq 3\delta\sigma\iint_{Q_T}e^{-2s\hat\alpha}{\hat\xi}^{-2}\abs{\Delta\psi}^2\dx\dt+C_{\delta}\sigma^{-1}s^4\iint_{Q_T}e^{-2s\hat\alpha}{\hat\xi}^2\abs{\psi}^2\dx\dt\\
	&\quad+C_{\delta}\sigma^{-1}s^3\iint_{Q_T}e^{-2s\hat\alpha}\hat\xi\abs{\phi}^2\dx\dt+C_{\delta}\sigma^{-1}s^3\iint_{Q_T}e^{-2s\hat\alpha}\hat\xi\abs{\psi}^2\dx\dt,
\end{aligned}
\end{equation}
for $\delta\in(0,1)$. If we choose $\delta$ sufficiently small, observe that in \eqref{e1.34} we can absorb the first term in the right-hand side to obtain
\begin{equation}\label{e1.35}
\begin{aligned}
	&\sigma\iint_{Q_T}e^{-2s\hat\alpha}{\hat{\xi}}^{-2}\abs{\Delta\psi}^2\dx\dt\leq C\sigma^{-1}s^4\iint_{Q_T}e^{-2s\hat\alpha}{\hat\xi}^2\abs{\psi}^2\dx\dt\\
	&\hspace{0.5cm}+C\sigma^{-1}s^3\iint_{Q_T}e^{-2s\hat\alpha}\hat\xi\abs{\phi}^2\dx\dt+C\sigma^{-1}s^3\iint_{Q_T}e^{-2s\hat\alpha}\hat\xi\abs{\psi}^2\dx\dt.
\end{aligned}
\end{equation}
Multiplying the inequality \eqref{e1.35} by $\sigma s^{-2}$ on both sides we get
\begin{equation}\label{e1.36}
\begin{aligned}
	&\sigma^2s^{-2}\iint_{Q_T}e^{-2s\hat\alpha}{\hat{\xi}}^{-2}\abs{\Delta\psi}^2\dx\dt\leq Cs^2\iint_{Q_T}e^{-2s\hat\alpha}{\hat\xi}^2\abs{\psi}^2\dx\dt\\
	&\hspace{0.5cm}+Cs\iint_{Q_T}e^{-2s\hat\alpha}\hat\xi\abs{\phi}^2\dx\dt+Cs\iint_{Q_T}e^{-2s\hat\alpha}\hat\xi\abs{\psi}^2\,dx\,dt,
\end{aligned}
\end{equation}
and using the properties of the functions defined in \eqref{e1.2} and \eqref{e1.3} we have that $e^{-2s\hat\alpha}\leq e^{-2s\alpha}$  and $\hat\xi\leq\xi$ which in combination with \eqref{e1.36}, yield
\begin{equation}\label{e1.37}
\begin{aligned}
	&\sigma^2s^{-2}\iint_{Q_T}e^{-2s\hat\alpha}{\hat{\xi}}^{-2}\abs{\Delta\psi}^2\,dx\,dt\leq Cs^2\iint_{Q_T}e^{-2s\alpha}\xi^2\abs{\psi}^2\,dx\,dt\\
	&\hspace{0.5cm}+Cs\iint_{Q_T}e^{-2s\alpha}\xi\abs{\phi}^2\,dx\,dt+Cs\iint_{Q_T}e^{-2s\alpha}\xi\abs{\psi}^2\,dx\,dt. 
\end{aligned}
\end{equation}

\textbf{Step 4: conclusion.} Recalling \eqref{e1.4}, we put put together estimates \eqref{e1.14}, \eqref{e1.31} and \eqref{e1.37}. Then taking $\delta \in (0,1)$ small enough and choosing the parameter $s$ as in \eqref{e1.13} we get
\begin{equation*}\label{e1.39}
\begin{aligned}
	s^3\iint_Q e^{-2s\alpha}\xi^3\abs{\phi}^2\,dx\,dt&+s^3\iint_Q e^{-2s\alpha}\xi^3\abs{\psi}^2\,dx\,dt \\ &\leq Cs^8\iint_{(0,T)\times\omega}e^{-4s\alpha^*+2s\hat\alpha}\pts{\xi^*}^8\abs{\phi}^2,
\end{aligned}
\end{equation*}
which is the desired result. 
\end{proof}

Now, we are in position to prove \Cref{Desigualdad-obs}.

\begin{proof}[Proof of \Cref{Desigualdad-obs}]
From \Cref{p1.2}, we have that for any $(\phi_T,\psi_T)\in [L^2(\Omega)]^2$, the solutions $(\phi,\psi)$ of \eqref{eq:adjoint} verify
\begin{equation}\label{eq:car_ineq_final}
\begin{aligned}
	s^3\iint_{Q_T} e^{-2s\alpha}\xi^3\abs{\phi}^2\dx\dt&+s^3\iint_{Q_T} e^{-2s\alpha}\xi^3\abs{\psi}^2\dx\dt \\ &\leq Cs^8\iint_{(0,T)\times\omega}e^{-4s\alpha^*+2s\hat\alpha}\pts{\xi^*}^8\abs{\phi}^2\dx\dt,
\end{aligned}
\end{equation}
for any 
\begin{equation}\label{eq:s_final}
	s\geq C\pts{T+T^2+T^2\max_{i,j=1,2}\nor{a_{ij}}_{L^\infty}^{\frac{2}{3}}}.
\end{equation}
where $C>0$ is a constant independent of $T$ and $a_{ij}$. 

Now, our goal is to bound the exponential weights in both sides of \eqref{eq:car_ineq_final}. To simplify the notation, we introduce the following
\begin{equation*}
\begin{aligned}
	&\tilde{M}_0:=\max_{x\in \bar\Omega}\pts{e^{2\lambda\nor{\eta^0}_{L^{\infty}}}-e^{\lambda\eta^0(x)}},\quad M_0:=\max_{x\in \bar\Omega}e^{\lambda\eta^0(x)}, \\
	&\tilde{m}_0:=\min_{x\in \bar\Omega}\pts{e^{2\lambda\nor{\eta^0}_{L^{\infty}}}-e^{\lambda\eta^0(x)}}, \quad m_0:=\min_{x\in \bar\Omega}e^{\lambda\eta^0(x)}.\\
\end{aligned}
\end{equation*}

Let $s_0:=\max\{C,\frac{1}{2^4m_0}\}$ where $C$ is the constant appearing in \eqref{eq:s_final} and fix 
\begin{equation}\label{eq:s_0}
s=s_0(T+T^2+T^2+\max_{i,j=1,2}\nor{a_{ij}}_{L^\infty}^{\frac{2}{3}}).
\end{equation}
We prove first that
\begin{equation}\label{e1.41}
	s^8(\xi^*)^8e^{-4s\alpha^*+2s\hat\alpha}\leq \pts{\frac{2^3M_0}{\tilde{m}_0e}}^8,
\end{equation} 
in $(0,T)$. Applying  \eqref{e1.40} with $\epsilon=\frac{1}{2}$ to the left-hand term of \eqref{e1.41} we see that
\begin{align} \notag
	s^8{\xi^*}^8e^{-4s\alpha^*+2s\hat\alpha}&\leq s^8{\xi^*}^8e^{-\alpha^*s}=s^8\frac{M_0^8}{t^8\pts{t-T}^8}e^{-\frac{\tilde{m}_0}{t\pts{T-t}}s}\\ \notag
	&\leq s^8\max_{t\in\pts{0,T}}\lvs{\frac{M_0^8}{t^8\pts{t-T}^8}e^{-\frac{\tilde{m}_0}{t\pts{T-t}}s}}
	=s^8M_0^8\frac{2^{16}}{T^{16}}e^{-\frac{2^2\tilde{m}_0}{T^2}s}\\ \label{eq:est_above_weight}
	&\leq M_0^8\frac{2^{16}}{T^{16}}\max_{s>0}\lvs{s^8e^{-\frac{2^2\tilde{m}_0}{T^2}s}}
	= \pts{\frac{2^3M_0}{\tilde{m}_0e}}^8,
\end{align}
where we have used that $\max_{x>0}x^a e^{-bx}=\left(\frac{a}{eb}\right)^{a}$ at $x=\frac{a}{b}$. 

On the other hand
\begin{equation}\label{e1.43}
	 s^3\xi^3e^{-2s\alpha}\geq\frac{1}{3^3}e^{-\frac{Cs}{T^2}} \quad \text{in }\pts{\frac{T}{4},\frac{3T}{4}}\times\Omega.
\end{equation}
Indeed,
\begin{align}\notag 
	 s^3\xi^3e^{-2s\alpha}&\geq \frac{s^3m_0^3}{t^3\pts{T-t}^3}e^{-\frac{2\tilde{M}_0}{t\pts{T-t}}s}\\ \notag
	 &\geq s^3m_0^3\min_{t\in\cts{\frac{T}{4},\frac{3T}{4}}}\lvs{\frac{1}{t^3\pts{t-t}^3}e^{-\frac{2\tilde{M}_0}{t\pts{T-t}}s}}\\ \label{e1.42}
	 &=s^3\pts{\frac{2^4m_0}{3T^2}}^3e^{-\frac{2^5\tilde{M}_0}{3T^2}s}.
\end{align}
From the particular choice of $s$ in \eqref{eq:s_0}, $s\geq\frac{T^2}{2^4m_0}$ and this yields \eqref{e1.43}. Using the bounds \eqref{eq:est_above_weight} and \eqref{e1.42} in \eqref{eq:car_ineq_final} gives
\begin{align}\label{eq.63}
\iint\limits_{(T/4,3T/4)\times\Omega} \abs{\phi}^2\dx\dt&+\iint\limits_{(T/4,3T/4)\times\Omega} \abs{\psi}^2\dx\dt \leq Ce^{\frac{Cs}{T^2}}\iint_{(0,T)\times\omega}\abs{\phi}^2\dx\dt,
\end{align}
for some $C>0$ independent of $T$ and $a_{ij}$.

Now, multiplying the first equation of \eqref{eq:adjoint} by $\phi$ and integrating over $\Omega$ (resp. multiplying the second equation of \ref{eq:adjoint} by $\psi$ and integrating over $\Omega$), we can deduce after some direct computations that
\begin{align*}\notag 
-\frac{d}{dt}&\int_{\Omega}\pts{\abs{\phi}^2+\abs{\psi}^2}\dx+2\int_{\Omega} \pts{\abs{\nabla\phi}^2+\abs{\nabla\psi}^2}\dx \\ %\label{eq.64a}
&\leq2\sum_{i,j}\nor{a_{ij}}_{L^\infty}\int_{\Omega}\pts{\abs{\phi}^2+\abs{\psi}^2}\dx.
\end{align*}
Then
\begin{equation}\label{eq.64}
\frac{d}{dt}\pts{e^{2\sum_{i,j}\nor{a_{ij}}_{L^\infty}t}\int_{\Omega}\pts{\abs{\phi}^2+\abs{\psi}^2}\dx}\geq 0.
\end{equation}
Let us now integrate this inequality on $\cts{\frac{T}{4},t}$ with $t\in\cts{T/4,3T/4}$
\begin{align*}
\int_{\Omega}&\left(\abs{\phi}^2+\abs{\psi}^2\right)\dx\\
&\geq e^{-\sum_{i,j}T\nor{a_{ij}}_{L^\infty}}\int_{\Omega}\left(\abs{\phi(T/4,x)}^2+\abs{\psi(T/4,x)}^2\right)\dx.
\end{align*}
Integrating on $\cts{T/4,3T/4}$ and after some straightforward estimates we have
\begin{align}\notag \int_{\Omega}&\left(\abs{\phi(T/4,x)}^2+\abs{\psi(T/4,x)}^2\right)\dx\\ \label{eq.66} 
    &\leq Ce^{C\pts{T^{-1}+T\sum_{i,j}\nor{a_{ij}}_{L^\infty}}}\iint\limits_{(T/4,3T/4)\times\Omega}\left(\abs{\phi}^2+\abs{\psi}^2\right)\dx.
\end{align}
On the other hand, using again the inequality \eqref{eq.64} and integrating on $\cts{0,T/4}$ we have
\begin{align} \notag 
\int_{\Omega}&\left(\abs{\phi(0,x)}^2+\abs{\psi(0,x)}^2\right)\dx\\ \label{eq.67}
& \leq e^{CT\sum_{i,j}\nor{a_{ij}}_{L^\infty}}\int\limits_{\Omega}\left(\abs{\phi(T/4,x)}^2+\abs{\psi(T/4,x)}^2\right) \dx.
\end{align}
Finally, combining the estimates \eqref{eq.63}, \eqref{eq.66} and \eqref{eq.67}, we obtain the observability inequality \eqref{e1.44}.
\end{proof}

\subsection{Controllability of the linear system with explicit control cost}\label{seccion-4}%\label{Odwlm}
In order to find a null control for \eqref{eq:r-d_controlled} with explicit control cost and uniform with respect to diffusion coefficient $\sigma$, we follow the classical penalized Hilbert Uniqueness Method (HUM) (see e.g. \cite{GLH08} or \cite{Boy13}). 

We begin by stating a uniform energy estimate  (w.r.t. $\sigma$) for the solutions of system \eqref{eq:r-d_controlled}. The result reads as follows. 
\begin{lem}\label{Lem:energy-estimate-l}
Let $\sigma\geq 1$, $(y^0,z^0)\in [L^2(\Omega)]^2$ and $h\in L^2(0,T;L^2(\omega))$ be given. There exists a constant $C>0$ depending only $\Omega$ such that $\pts{y,z}$ the solution of system \eqref{eq:r-d_controlled} satisfies the following estimate 
\begin{equation}\label{energy-estimate-l}
\begin{aligned}
	&\nor{y}_{L^2\pts{0,T;H^1(\Omega)}}^2+\nor{z}_{L^2\pts{0,T;H^1(\Omega)}}^2\\
	&+\nor{y_t}_{L^2(0,T;(H^1(\Omega))')}^2+\nor{z_t}_{L^2(0,T;(H^1(\Omega))')}^2+ \sigma\iint_{Q_T}\abs{\nabla z}^2\,dx\,dt\\
	&\hspace{3cm}\leq e^{C\tilde{K}}\pts{\nor{y^0}_{L^2(\Omega)}^2+\nor{z^0}_{L^2(\Omega)}^2+\nor{h}_{L^2(0,T;L^2(\omega))}^2},
\end{aligned}
\end{equation}
where
\begin{equation}\label{Constante-estimado-energia}
    \Tilde{K}=\left(\sum_{i,j=1}^2\nor{a_{ij}}_{L^{\infty}(Q_T)}^2+1\right)(T+1).
\end{equation}
\end{lem}

For the proof, we can follow a classical methodology (see e.g. \cite[Section 7.1]{evans10}), just by taking care that in each step the estimates are independent of $\sigma$. For brevity, we omit it. 

Now, we are in position to prove the following. 
\begin{theorem}\label{thm:Control-lineal}
Under the assumptions of Proposition~\ref{Desigualdad-obs}, for every $(y^0,z^0)\in [L^2(\Omega)]^2$ there exists a control $h\in L^2(0,T;L^2(\omega))$ such that the solution $(y,z)$ of the system \eqref{eq:r-d_controlled} satisfies $y(T,\cdot)=z(T,\cdot)=0$ in $\Omega$. Moreover, we have that the control cost is given by 
\begin{equation}\label{Estimate-control-1}
    \nor{ h}_{L^2(0,T;L^2(\omega))}\leq e^{CK}\pts{\nor{y^0}_{L^2(\Omega)}^2+\nor{z^0}_{L^2(\Omega)}^2},
\end{equation}
where $C>0$ is a constant independent of $\sigma$ and $K$ is defined in \eqref{e1.45}.
\end{theorem}

\begin{proof}

For readibility, we have divided the proof in three steps.

\smallskip
\textbf{Step 1:} Let $\sigma\geq 1$ be fixed and consider the following minimization problem
\begin{equation*}
	\min_{h\in L^2((0,T)\times\omega)}F_{\varepsilon}(h)
\end{equation*} 
where, for every $\epsilon>0$, we write
\begin{equation}\label{e2.69}
	F_{\varepsilon}(h)=\frac{1}{2}\iint\limits_{(0,T)\times\omega}\abs{h}^2\dx\dt+\frac{1}{2\varepsilon}\pts{\int_{\Omega}\abs{y(T)}^2\dx+\int_{\Omega}\abs{z(T)}^2\dx}.
\end{equation} 
Since $F_{\varepsilon}$ is continuous, strictly convex, and coercive, there exists a unique minimizer that we denote by $h_{\varepsilon}$. By a classical procedure (see, for instance, \cite{Lio71}), i.e., obtaining the Euler-Lagrange equation for \eqref{e2.69} at the minimum $h_\varepsilon$ and a duality argument, the control $h_\varepsilon$ can be characterized as
\begin{equation}\label{eq:def_h_eps}
h_\varepsilon=-\phi_{\varepsilon}|_{\omega} \quad\text{in $Q_T$}
\end{equation}
where $\phi_\varepsilon$ is the solution of the first component of $(\phi_\varepsilon,\psi_\varepsilon)$ verifying 
\begin{equation*}%\label{eq:adjoint-epsilon}
\begin{cases}
    -{\phi_{\varepsilon}}_t-\Delta\phi_{\varepsilon}=a_{11}\phi_{\varepsilon}+a_{21}\psi_{\varepsilon} &\textnormal{in } Q_T, \\
    -{\psi_{\varepsilon}}_t-\sigma\Delta\psi_{\varepsilon}=a_{12}\phi_{\varepsilon}+a_{22}\psi_{\varepsilon} &\textnormal{in } Q_T, \\
    \dfrac{\partial \phi_{\varepsilon}}{\partial \nu}=\dfrac{\partial \psi_{\varepsilon}}{\partial \nu}=0 &\textnormal{on } \Sigma_T, \\
    \phi_{\varepsilon}(T,\cdot)=\varepsilon^{-1}y_{\varepsilon}(T),\quad \psi_{\varepsilon}(T,\cdot)=\varepsilon^{-1}z_{\varepsilon}(T) &\textnormal{in } \Omega,
\end{cases}
\end{equation*}
and where $(y_\varepsilon(T),z_{\varepsilon}(T))$ can be extracted from the solution to system \eqref{eq:r-d_controlled} with control $h=h_\varepsilon$, more precisely, 
\begin{equation}\label{eq:r-d_controlled-epsilon}
    \begin{cases}
    {y_{\varepsilon}}_t-\Delta y_{\varepsilon} = a_{11} y_{\varepsilon} + a_{12} z_{\varepsilon} +h_{\varepsilon}\mathbf{1}_{\omega} &\textnormal{in } Q_T, \\
    {z_{\varepsilon}}_t-\sigma\Delta z_{\varepsilon}= a_{21} y_{\varepsilon} + a_{22} z_{\varepsilon} &\textnormal{in } Q_T, \\
    \D \frac{\partial y_{\varepsilon}}{\partial \nu}=\frac{\partial z_{\varepsilon}}{\partial \nu}=0 &\text{on } \Sigma_T, \\
    y_{\varepsilon}(0,\cdot)=y^0,\quad z_{\varepsilon}(0,\cdot)=z^0  &\text{in }\ \Omega.
    \end{cases}
\end{equation}

Let us prove the following convergences
\begin{equation}\label{e2.89}
\begin{aligned}
	&h_{\varepsilon}\rightharpoonup h \quad\text{in}\quad L^2((0,T)\times\omega), \\
	&(y_{\varepsilon}(T), z_{\varepsilon}(T))\rightarrow (0,0) \quad\text{in}\quad [L^2(\Omega)]^2,
\end{aligned}
\end{equation}
for some $h$ in $L^2((0,T)\times\omega)$. Indeed, by duality between the solutions of $(y_{\varepsilon},z_{\varepsilon})$ and $(\phi_{\varepsilon},\psi_{\varepsilon})$, we have 
\begin{equation}\label{e2.96}
\small
\begin{aligned}
	&\frac{1}{\varepsilon}\int_{\Omega} \pts{\abs{y_{\varepsilon}(T)}^2+\abs{z_{\varepsilon}(T)}^2}\dx+\iint\limits_{(0,T)\times\omega}\abs{\phi_{\varepsilon}}^2\dx\dt=\int_{\Omega}\pts{y^0\phi_{\varepsilon}(0)+z^0\psi_{\varepsilon}(0)}\dx.
\end{aligned}
\end{equation}
Applying Cauchy-Schwarz inequality and using estimate \eqref{e1.44} in the right-hand side of \eqref{e2.96} we see that
\begin{equation}\label{e2.98}
\small
\begin{aligned}
	&\int_{\Omega}\pts{y^0\phi_{\varepsilon}(0)+z^0\psi_{\varepsilon}(0)}\,dx\leq \nor{\pts{y^0,z^0}}_{\pts{L^2(\Omega)}^2}\pts{e^{CK}\iint\limits_{(0,T)\times\omega}\abs{\phi_{\varepsilon}}^2\dx\dt}^{\frac{1}{2}}.
\end{aligned}
\end{equation}
and combining \eqref{e2.96}, \eqref{e2.98}, and applying Young's inequality with $\delta>0$ to the right-hand side of \eqref{e2.98} we get
\begin{equation}\label{e2.101}
\small
\begin{aligned}
	\frac{1}{\varepsilon}\int_{\Omega} \pts{\abs{y_{\varepsilon}(T)}^2+\abs{z_{\varepsilon}(T)}^2}\dx&+\iint\limits_{(0,T)\times\omega}\abs{h_{\varepsilon}}^2\dx\dt \\
    &\leq e^{CK}\nor{\pts{y^0,z^0}}_{\pts{L^2(\Omega)}^2}^2 +\delta\iint\limits_{(0,T)\times\omega}\abs{h_{\varepsilon}}^2\dx,
\end{aligned}
\end{equation}
where we have used \eqref{eq:def_h_eps}. Taking $\delta>0$  sufficiently small, we can eliminate the local term in the right-hand side of \eqref{e2.101} to obtain
\begin{equation}\label{e2.102a}
\small
\begin{aligned}
	\frac{1}{\varepsilon}\int_{\Omega} \pts{\abs{y_{\varepsilon}(T)}^2+\abs{z_{\varepsilon}(T)}^2}\,dx&+\iint\limits_{(0,T)\times\omega}\abs{h_{\varepsilon}}^2\,dxdt \leq e^{CK}\nor{\pts{y^0,z^0}}_{\pts{L^2(\Omega)}^2}^2.
\end{aligned}
\end{equation}
Note that this estimate is independent of the diffusion coefficient $\sigma$ and the parameter $\varepsilon$. Therefore, the estimate \eqref{e2.102a} implies the convergences of \eqref{e2.89}. 
To conclude this step, we obtain a uniform bound on the control $h$ with respect to $\sigma$ as follows. Since 
\begin{equation}\label{eq:e.85}
    h_{\varepsilon}\rightharpoonup h \quad\text{in}\quad L^2((0,T)\times\omega),
\end{equation}
holds in \eqref{e2.89} we can apply Fatou's Lemma  (see \cite[Colloraly~II.2.8]{Boyer12}) to get
\begin{equation}\label{Estimate-Fatou}
    \nor{h}_{L^2(0,T;L^2(\omega))}\leq \liminf_{\varepsilon\rightarrow 0}\nor{h_{\varepsilon}}_{L^2(0,T;L^2(\omega))}.
\end{equation}
Using \eqref{e2.102a} in the right-hand side of \eqref{Estimate-Fatou} we obtain
\begin{equation}\label{Estimate-control}
    \nor{h}_{L^2(0,T;L^2(\omega))}\leq e^{CK}\pts{\nor{y^0}_{L^2(\Omega)}^2+\nor{z^0}_{L^2(\Omega)}^2},
\end{equation}
where $C$ is independent of $\sigma$, $\varepsilon$ and where $K$ is defined in \eqref{e1.45}. This proves the control estimate in \eqref{Estimate-control-1}.

\smallskip
\textbf{Step 2:} Here we prove that for any $\sigma\geq 1$ fixed
 \begin{equation}\label{e2.102}
\begin{aligned}
	&y_{\varepsilon}\rightharpoonup y \quad \text{en} \quad L^2\pts{0,T;H^1(\Omega)}, \\
	&z_{\varepsilon}\rightharpoonup z \quad \text{en} \quad L^2\pts{0,T;H^1(\Omega)},
\end{aligned}
\end{equation}
uniformly with respect $\varepsilon>0$, where $(y,z)\in L^2(0,T;H^1(\Omega))$ satisfies \eqref{eq:r-d_controlled} with the control $h$ defined in Step 1. 

To this end, using Lemma \ref{Lem:energy-estimate-l}, the solution $(y_{\varepsilon},z_{\varepsilon})$ to \eqref{eq:r-d_controlled-epsilon} satisfies 
\begin{equation*}%\label{e2.103}
\begin{aligned}
&\nor{y_{\varepsilon}}_{L^2\pts{0,T;H^1(\Omega)}}^2+\nor{z_{\varepsilon}}_{L^2\pts{0,T;H^1(\Omega)}}^2\\
&\quad \leq e^{C\tilde{K}}\pts{\nor{y^0}_{L^2(\Omega)}^2+\nor{z^0}_{L^2(\Omega)}^2+\nor{h_{\varepsilon}}_{L^2(0,T;L^2(\omega))}^2}, 
\end{aligned}
\end{equation*}
where $C>0$  is independent of $\sigma$ and $\varepsilon$, and where we recall that $\Tilde{K}$ is defined in \eqref{Constante-estimado-energia}. Using \eqref{e2.102a} on the right-hand of the above expression yields 
\begin{equation*}%\label{e2.103a}
\begin{aligned}	&\nor{y_{\varepsilon}}_{L^2\pts{0,T;H^1(\Omega)}}^2+\nor{z_{\varepsilon}}_{L^2\pts{0,T;H^1(\Omega)}}^2\leq e^{C\bar{K}}\pts{\nor{y^0}_{L^2(\Omega)}^2+\nor{z^0}_{L^2(\Omega)}^2}, 
\end{aligned}
\end{equation*}
with 
\begin{equation*}%\label{e2.103b}
\begin{aligned} \bar{K}&=1+T^{-1}+T\pts{1+\sum_{i,j=1}^2\nor{a_{ij}}_{L^{\infty}}+\sum_{i,j=1}^2\nor{a_{ij}}_{L^{\infty}}^2}\\
&\quad +\max_{i,j=1,2}\nor{a_{ij}}_{L^{\infty}}^{\frac{2}{3}}+\sum_{i,j=1}^2\nor{a_{ij}}_{L^{\infty}}^2.
\end{aligned}
\end{equation*}
Therefore, we can extract a subsequence (still denoted) $\lvs{(y_\varepsilon,z_{\varepsilon})}_{\varepsilon\geq 0}$ such that it verifies the weak convergences in \eqref{e2.102}. To check that the limits in \eqref{e2.102} satisfy \eqref{eq:r-d_controlled} with the control $h$ obtained in the previous step, we argue as follows.

 Let us denote by $(\hat{y},\hat{z})$ the solution to 
\begin{equation}\label{eq:r-d_controlled-gorro}
    \begin{cases}
    \hat y_t-\Delta \hat y = a_{11} \hat y + a_{12} \hat z +h\mathbf{1}_{\omega} &\textnormal{in } Q_T, \\
    \hat z_t-\sigma\Delta \hat z= a_{21} \hat y + a_{22} \hat z &\textnormal{in } Q_T, \\
    \D \frac{\partial \hat y}{\partial \nu}=\frac{\partial \hat z}{\partial \nu}=0 &\text{on } \Sigma_T, \\
    \hat y(0,\cdot)=y^0,\quad \hat z(0,\cdot)=z^0  &\text{in }\ \Omega, 
    \end{cases}
\end{equation}
where $h$ is a control provided by \eqref{e2.89}. For any $(F_1,F_2)\in [L^2(Q_T)]^2$ and any $(\Phi_T,\Psi_T)\in [L^2(\Omega)]^2$, let us introduce the following adjoint system
\begin{equation}\label{eq:adjoint-gen}
\begin{cases}
    -\Phi_t-\Delta\Phi=a_{11}\Phi+a_{21}\Psi+F_1 &\textnormal{in } Q_T, \\
    -\Psi_t-\sigma\Delta\Psi=a_{12}\Phi+a_{22}\Psi+F_2 &\textnormal{in } Q_T, \\
    \dfrac{\partial \Phi}{\partial \nu}=\dfrac{\partial \Psi}{\partial \nu}=0 &\textnormal{on } \Sigma_T,\\
    \Phi(T,\cdot)=\Phi_T,\quad \Psi(T,\cdot)=\Psi_T &\textnormal{in } \Omega.
\end{cases}
\end{equation}

Setting $(\Phi_T,\psi_T)=(0,0)$ in \eqref{eq:adjoint-gen}, we obtain by duality between \eqref{eq:r-d_controlled-gorro} and \eqref{eq:adjoint-gen} that
\begin{equation}\label{e2.122}
\begin{aligned}
	&-\int_{\Omega} \pts{y^0\Phi(0)+z^0\Psi(0)}\dx=-\iint_{Q_T}\pts{\hat yF_1+\hat zF_2}\dx\dt+\iint_{Q_T}h\Phi\,\dx\dt. 
\end{aligned}
\end{equation}
In the same spirit, from \eqref{eq:r-d_controlled-epsilon}, \eqref{eq:adjoint-gen}, and recalling that we have set zero initial data, we have 
\begin{equation}\label{e2.123}
\begin{aligned}
	&-\int_{\Omega} \pts{y^0\phi(0)+z^0\psi(0)}\dx=-\iint_{Q_T}\pts{y_{\varepsilon}F_1+z_{\varepsilon}F_2}\dx\dt+\iint_{Q_T}h_{\varepsilon}\phi\, \dx\dt. 
\end{aligned}
\end{equation}
From \eqref{eq:e.85} and \eqref{e2.102}, we can pass to the limit as $\varepsilon\to 0$ in  \eqref{e2.123} and this, together with \eqref{e2.122}, yields
\begin{equation*}%\label{e2.127}
\begin{aligned}
	\iint_{Q_T}\pts{y-\hat y}F_1\dx\dt+\iint_{Q_T}\pts{z-\hat z}F_2\, \dx\dt=0,
\end{aligned}
\end{equation*}
for all $\pts{F_1,F_2}\in [L^2(Q_T)]^2$.  This implies that $(y,z)=(\hat y,\hat{z})$ which proves our initial claim. 

\textbf{Step 3:} 
To conclude, let us see that the limit $h$ obtained in Step 1 is in fact a null control for \eqref{eq:r-d_controlled}. From the conclusion of Step 2, by duality between $(y,z)$ solution to \eqref{eq:r-d_controlled} and \eqref{eq:adjoint-gen} with $(F_1,F_2)=(0,0)$, we deduce
\begin{equation}\label{e2.132}
	\int_{\Omega}\pts{y(T)\Phi_T+z(T)\Psi_T}\dx=\iint\limits_{(0,T)\times\omega}h\Phi\,\dx\dt+\int_{\Omega}\pts{y^0\Phi(0)+z^0\Phi(0)}\dx\dt
\end{equation}
Similarly, we can readily check that the solutions $(y_{\varepsilon},z_{\varepsilon})$ and $(\Phi, \Psi)$  to \eqref{eq:r-d_controlled-epsilon} and \eqref{eq:adjoint-gen} (with $(F_1,F_2)=(0,0)$) respectively, verify 
\begin{equation}\label{e2.133}
	\int_{\Omega}\pts{y_{\varepsilon}(T)\Phi_T+z_{\varepsilon}(T)\Psi_T}\dx=\iint_{Q_T}h_{\varepsilon}\Phi\,\dx\dt+\int_{\Omega}\pts{y^0\Phi(0)+z^0\Phi(0)}\dx\dt. 
\end{equation}

Recalling the convergences provided in \eqref{e2.89} and passing the limit as  $\varepsilon\rightarrow 0$ in \eqref{e2.133} yield
\begin{equation}\label{e2.135}
	\iint\limits_{(0,T)\times\omega}h\Phi\,\dx\dt+\int_{\Omega}\pts{y^0\Phi(0)+z^0\Phi(0)}\dx\dt=0.
\end{equation}
Substituting \eqref{e2.135} in \eqref{e2.132} gives
\begin{equation*}%\label{e2.136}
	\int_{\Omega}\pts{y(T)\Phi_T+z(T)\Psi_T}\dx=0,
\end{equation*}
for all $\pts{\Phi_T,\Psi_T}\in [L^2(\Omega)]^2$. Then $y(T,\cdot)=z(T,\cdot)=0$ in $\Omega$ as claimed. This ends the proof.
\end{proof}

\section{Controllability of the nonlinear system: proof of \Cref{T1}}\label{seccion-5}
To prove the \Cref{T1}, we begin by considering a linearized system. Applying Taylor formula to the nonlinearities $f$ and $g$, we have that for every $\sigma\geq1$ fixed and each $(\bar{y},\bar{z})\in [L^2(Q_T)]^2$, system \eqref{eq:r-d_controlled-NL} can be expressed as
\begin{equation}\label{e2.160}
\begin{cases}
	y_t-\Delta y=a_{11}y+a_{12}z +h1_{\omega} &\text{in }  Q_T \\
	z_t-\sigma\Delta z=a_{21}y+a_{22}z & \text{in }  Q_T \\
	\frac{\partial y}{\partial \hat\eta}=\frac{\partial z}{\partial \hat\eta}=0, &\text{on} \Sigma_T \\
	y(0,\cdot)=y^0,\quad z(0,\cdot)=z^0, & \text{in } \Omega.
\end{cases}	
\end{equation}
where
\begin{equation*}%\label{e2.161}
\begin{aligned}
	&a_{11}=\int_0^1\frac{\partial f}{\partial y}\pts{\delta\bar{y},\delta\bar{z}}\dd \delta, \quad &a_{12}=\int_0^1\frac{\partial f}{\partial z}\pts{\delta\bar{y},\delta\bar{z}}\dd \delta,\\
	&a_{21}=\int_0^1\frac{\partial g}{\partial y}\pts{\delta\bar{y},\delta\bar{z}}\dd \delta, \quad &a_{22}=\int_0^1\frac{\partial g}{\partial z}\pts{\delta\bar{y},\delta \bar{z}}\dd \delta.
\end{aligned}	
\end{equation*}

We note that the assumption \ref{H3} implies that there exist a constant $\hat{a}_{21}>0$ such that $a_{21}(t,x)\geq\hat{a}_{21}$ or $-a_{21}(t,x)>\hat{a}_{21} $ for all $(t,x)\in(0,T)\times\omega$. Furthermore, by  assumption \ref{H1} there exist a constant $C$ that only depending on $T$, $\Omega$, $C_f$ and $C_g$ such that 
\begin{equation}\label{coeff}
    \nor{a_{ij}}_{L^\infty(Q_T)}\leq C,
\end{equation}
for $i,j=1,2$. 

Therefore the hypotheses of Theorem \ref{thm:Control-lineal} are satisfied and we can build a control $h$ such that the solution $(y,z)$ to \eqref{e2.160} satisfy 
\begin{align*}
y(T,x;\bar y,\bar z)=z(T,x;\bar y,\bar z)=0, \quad x\in\Omega,
\end{align*}
for each $\sigma\geq 1$ and each $(\bar{y},\bar{z})\in [L^2(Q_T)]^2$. By construction, such control is uniformly bounded with respect to $\sigma$, and thanks to \eqref{coeff} it is also uniformly bounded with respect to $(\bar{y},\bar{z})$.

Now, we consider the following map $\Lambda:[L^2(Q_T)]^2 \rightarrow [L^2(Q_T)]^2$ given by
\begin{equation*}
	\Lambda(\bar{y},\bar{z})=\pts{y(t,x,\bar{y},\bar{z}),z(t,x,\bar{y},\bar{z})},
\end{equation*} 
where $\pts{\bar{y},\bar{z}}\in[L^2(Q_T)]^2$ and $(y,z)$ is solution of \eqref{e2.160}. Due to the Aubin--Lions Lemma, the space
\begin{equation*}
	W=\left\{u:u\in L^2\pts{0,T;H^1(\Omega)}, \; u_t\in L^2\pts{0,T;\pts{H^1(\Omega)}'}\right\},  
\end{equation*}
is compactly embedded into $L^2(Q_T)$. This, together with the energy estimate given in Lemma~\ref{Lem:energy-estimate-l}, yields that $\Lambda$ is continuous and compact mapping in $[L^2\pts{Q_T}]^2$ into itself. Also, using the Lemma~\ref{Lem:energy-estimate-l} we can see that the set
\begin{equation*}
	M=\left\{(y,z)\in[L^2(Q_T)]^2:{(y,z)=\lambda \Lambda\pts{y,z}\quad\text{for some}\quad0\leq\lambda\leq1} \right\},
\end{equation*}
is bounded in $[L^2(Q_T)]^2$. Indeed, for every $\pts{y,z}\in M$, according to the energy estimate \eqref{energy-estimate-l} and the uniform bound on the control $h$ with respect to $\sigma\geq1$ given in \eqref{Estimate-control-1} we have that
\begin{equation*}
\begin{aligned}
    \nor{(y,z)}_{[L^2(Q_T)]^2}^2&=\nor{\lambda\Lambda(y,z)}_{[L^2(Q_T)]^2}^2\\
    &=\lambda^2\nor{(y,z)}_{[L^2(Q_T)]^2}^2\\
    &\leq e^{C\tilde{K}}\pts{\nor{y^0}_{L^2(\Omega)}^2+\nor{z^0}_{L^2(\Omega)}^2+\nor{h}_{L^2(0,T;L^2(\omega))}^2}\\
    &\leq e^{C\bar{K}}\pts{\nor{y^0}_{L^2(\Omega)}^2+\nor{z^0}_{L^2(\Omega)}^2}\\
    &\leq C\pts{\nor{y^0}_{L^2(\Omega)}^2+\nor{z^0}_{L^2(\Omega)}^2},
\end{aligned}
\end{equation*}
where in last line we have used \eqref{coeff} to obtain that the constant $C$ only depending on $T$, $\Omega$, $C_f$ and $C_g$. 

Therefore, the map $\Lambda$ satisfies all the hypotheses of Schaefer's fixed point theorem (see e.g. \cite[Section 9.2]{evans10}) and there is $(y,z)\in (L^2\pts{0,T;L^2(\Omega)})^2$ such that $\Lambda(y,z)=(y,z)$. Consequently, $y(T)=z(T))=0$ in $\Omega$. 
%For this, we can conclude that the semilinear reaction-diffusion system \eqref{eq:r-d_controlled} is null controllable.
Moreover, by construction is clear that $h=h(\sigma)$ is uniformly bounded as in \eqref{Estimate-control-1} and using \eqref{coeff} we obtain 
\begin{equation*}
\begin{aligned}
 	\nor{h(\sigma)}_{L^2((0,T)\times\omega)} \leq C\pts{{\nor{y^0}_{L^2(\Omega)}^2+\nor{z^0}_{L^2(\Omega)}^2}},
\end{aligned}
\end{equation*}
for every $\sigma\geq1$ fixed and some constant $C>0$ only depending on $\Omega$, $\omega$, $T$, $C_f$, and  $C_g$.

Since $(y,z)$ is solution of the linearized system \eqref{e2.160} for any $(\bar y,\bar z)\in [L^2(Q_T)]^2$ we have that satisfies the energy estimate \eqref{energy-estimate-l}. In particular, for the fixed point $(y,z)$ obtained above we get 
\begin{equation*}
\begin{aligned}
	&\nor{y}_{L^2\pts{0,T;H^1(\Omega)}}^2+\nor{z}_{L^2\pts{0,T;H^1(\Omega)}}^2\\
	&+\nor{y_t}_{L^2(0,T;(H^1(\Omega))')}^2+\nor{z_t}_{L^2(0,T;(H^1(\Omega))')}^2+ \sigma\iint_{Q_T}\abs{\nabla z}^2\,dx\,dt\\
	&\hspace{3cm}\leq e^{C\bar K}\pts{\nor{y^0}_{L^2(\Omega)}^2+\nor{z^0}_{L^2(\Omega)}^2} \\
	&\hspace{3cm}\leq C\pts{\nor{y^0}_{L^2(\Omega)}^2+\nor{z^0}_{L^2(\Omega)}^2},
\end{aligned}
\end{equation*}
where in the last line we have used again \eqref{coeff} to obtain that a constant $C$ only depending on $T$, $\Omega$, $\omega$, $C_f$, and $C_g$. Thus the proof of \Cref{T1} is finished.

%%%%%%%%%%%%%%%%%%%%%%%%%%%%%%%%%%%%%%%%%%%%%%%%%%%%%%%%%%%%%%%%%%%%%%%%%%%%%%%%%%%%%%%%%%%%%%%%%%%%%%%%%%%%%%%%%%%%%%%%%%%%%%%%%%%%%%%%%%%%%%%%%%%%%%%%%%%%%%%%%%%%%%%%%%%%%%%%%%%%%%%%%%%%%%%%%%%%%%
\section{Passing to the shadow limit: proof of \Cref{thm:main_1}} \label{seccion-6}
In this part, we assume without loss of generality that $\abs{\Omega}=1$ to simplify some computations. As before, for clarity we argue in several steps. 

\smallskip
{\bf Step 1:} Let us consider sequences $\lvs{y^{\sigma},z^{\sigma}}_{\sigma\geq 1}$ of solutions of \eqref{eq:r-d_controlled-NL} and controls $\lvs{h(\sigma)}_{\sigma\geq 1}$, such that the controls $h(\sigma)$ are provided by the Theorem~\ref{T1}. Therefore, $y^{\sigma}(T)=z^{\sigma}(T)=0$ in $\Omega$ for every $\sigma\geq1$. 

In this step, we will prove that  
\begin{equation}\label{strong-converge}
	(y^{\sigma},z^{\sigma})\rightarrow (y,\xi) \quad \text{in}\quad \quad L^2(Q_T)\times L^2(0,T),
\end{equation}
such that $y$ and $\xi$ satisfies the first equation of \eqref{eq:shadow_controller} in the weak sense. Moreover, we will check that $y(T)=0$. We argue as follows.  

For any $\sigma\geq1$, the solution $(y^{\sigma},z^{\sigma})$ to \eqref{eq:r-d_controlled-NL} satisfies the estimate \eqref{energy-estimate-sigma} and the control $h$ is uniformly bounded as in \eqref{eq:unif_bound}. Therefore $\lvs{y^{\sigma},z^{\sigma}}_{\sigma\geq 1}$ is a uniformly bounded sequence in $[L^2(0,T;H^1(\Omega))]^2$ (w.r.t. $\sigma$). Then,  we can extract a subsequence, still denoted $\lvs{y^{\sigma},z^{\sigma}}_{\sigma\geq 1}$ such that 
\begin{equation*}%\label{e3.193}
\begin{aligned}
	(y^{\sigma},z^{\sigma})\rightharpoonup (y,z) \quad \text{in}\quad (L^2(0,T;H^1(\Omega))^2.
\end{aligned}
\end{equation*}
Furthermore,  $\lvs{y_t^{\sigma},z_t^{\sigma}}_{\sigma\geq 1}$ is uniformly bounded in $(L^2(0,T;(H^1(\Omega))')^2$  with respect to the coefficient diffusion $\sigma\geq 1$. Thus, by Aubin-Lions Lemma, we get that
\begin{equation}\label{e3.194}
	(y^{\sigma},z^{\sigma})\rightarrow (y,z) \quad \text{in}\quad [L^2(Q_T)]^2.
\end{equation} 

According to Fatou's Lemma, using the strong convergence \eqref{e3.194} and arguing as in the proof of theorem~\ref{thm:Control-lineal}, we deduce 
\begin{equation}\label{e3.195}
	\nor{y}_{L^2(Q_T)}^2+\nor{z}_{L^2(Q_T)}^2\leq C\pts{\nor{y^0}_{L^2(\Omega)}^2+\nor{z^0}_{L^2(\Omega)}^2},
\end{equation}
where the constant $C>0$ is independent of the diffusion coefficient $\sigma\geq 1$. Since $f$ and $g$ are Lipschitz functions, we can use the strong convergences in \eqref{e3.194} to see that
\begin{equation*}%\label{e3.197}
\begin{aligned}
	&f(y^{\sigma},z^{\sigma})\rightarrow f(y,z) \quad \text{in}\quad L^2(Q_T),\\
	&g(y^{\sigma},z^{\sigma})\rightarrow g(y,z) \quad \text{in}\quad L^2(Q_T).
\end{aligned}
\end{equation*}

Moreover, from the energy estimate \eqref{energy-estimate-sigma} we know that $\lvs{\nabla y^{\sigma}}_{\sigma\geq 1}$ is uniformly bounded $L^2(Q_T)$ and 
\begin{equation*}
	\nor{\nabla y^{\sigma}}_{L^2(Q_T)}+\sigma\nor{\nabla z^{\sigma}}_{L^2(Q_T)}\leq C\pts{\nor{y^0}_{L^2(\Omega)}^2+\nor{z^0}_{L^2(\Omega)}^2},
\end{equation*} 
whence we deduce that 
\begin{equation}\label{e3.200}
\begin{aligned}
	&\nabla y^{\sigma}\rightharpoonup\nabla y \quad L^2(Q_T),\\
	&\nabla z^{\sigma}\rightarrow 0 \quad L^2(Q_T).
\end{aligned}
\end{equation} 
The second convergence in \eqref{e3.200} implies that the limit $z$ only depends on $t$, so we can write it as $z(t,\cdot)=\xi(t)$ for $t\in(0,T)$. Using this fact together with the convergence \eqref{e3.194} we obtain \eqref{strong-converge}.

Also, we recall that estimate \eqref{eq:unif_bound} says that the sequence of controls $\lvs{h(\sigma)}_{\sigma\geq 1}$ is uniformly bounded with respect $\sigma\geq 1$. Then we can extract a subsequence such that 
\begin{equation*}
	h(\sigma)\rightharpoonup h\text{ in } L^2(0,T;L^2(\omega)).
\end{equation*}
Therefore, the  convergences above imply that the $y$ and  $\xi$ satisfies the first equation of \eqref{eq:shadow_controller}, in the weak sense. Since $y^{\sigma}(T)=0$, we have that $y(T)=0$ in $\Omega$. It remains to prove that the solution $\xi$ satisfies the second equation of \eqref{eq:shadow_controller}. 

\smallskip

{\bf Step 2:} In this step we obtain a first estimation of the difference between $z(t,\cdot)$ and $\xi(t)$ in $L^2$-norm for $t\in(0,T)$. For this, we consider the heat semigroup with homogeneous Neumann boundary conditions on a domain $\Omega\subset\mathbb R^n$ denoted by $e^{t\sigma\Delta}$ for all $t\geq 0$ and write the solutions $z$ and $\xi$ of the second component of systems \eqref{eq:r-d_controlled-NL} and \eqref{eq:shadow_controller}, respectively, as follows 
\begin{equation*}%\label{e3.201}
\begin{aligned}
	&z^{\sigma}(t,x)=e^{t\sigma\Delta}z^0+\int_0^te^{(t-s)\sigma\Delta}g(y^{\sigma}(s,x),z^{\sigma}(s,x))\dd{s},\\
	&\xi(t)=\int_{\Omega}z^0\,dx+\int_0^t\int_{\Omega}g(y(s,x),\xi(s))\dx\dd{s}.
\end{aligned}
\end{equation*}
Subtracting the above solutions and computing the $L^2$-norm we get 
\begin{equation}\label{e3.202}
\begin{aligned}
	&\nor{z^{\sigma}(t,\cdot)-\xi(t)}_{L^2(\Omega)}\\
	&\leq\Bigg\|e^{t\sigma\Delta}z^0-\int_{\Omega}z^0\dx+\int_0^te^{(t-s)\sigma\Delta}\pts{g(y^{\sigma}(s,\cdot),z^{\sigma}(s,\cdot))}\dd{s}\\
    &\hspace{4.6cm}-\int_0^t\int_{\Omega}g(y(s,x),\xi(s,x))\dx\dd{s}\Bigg\|_{L^2(\Omega)}.
\end{aligned}
\end{equation}

Applying property $i)$ of Lemma~\ref{A.2} to the terms $c_1=\int_{\Omega}z^0\,dx$ and $c_2=\int_{\Omega}g(y(t,x),\xi(t))\dx$ in the right-hand side of \eqref{e3.202} we get
\begin{equation}\label{e3.203}
\begin{aligned}
	&\nor{z^{\sigma}(t,\cdot)-\xi(t)}_{L^2(\Omega)}\\
	&\leq\nor{e^{t\sigma\Delta }\pts{z^0-\int_{\Omega}z^0\dx}}_{L^2(\Omega)}\\
	&+\nor{\int_0^te^{(t-s)\sigma\Delta}\pts{g(y^{\sigma}(s,\cdot),z^{\sigma}(s,\cdot))-\int_{\Omega}g(y^{\sigma}(s,x),z^{\sigma}(s,x))}\dd{s}}_{L^2(\Omega)}\\
&+\nor{\int_0^te^{(t-s)\sigma\Delta}\int_{\Omega}g(y^{\sigma}(s,x),z^{\sigma}(s,x))-g(y(s,x),\xi(s,x))\dx\dd{s}}_{L^2(\Omega)}.
\end{aligned}
\end{equation}
We apply once again property $i)$ of Lemma~\ref{A.2} to the term $c_3=\int_{\Omega}g(y^{\sigma}(s,x),z^{\sigma}(s,x))-g(y(s,x),\xi(s,x))\dx$ in the right-hand side of \eqref{e3.203} to obtain
\begin{equation}\label{e3.204}
\begin{aligned}
	&\nor{z^{\sigma}(t,\cdot)-\xi(t)}_{L^2(\Omega)}\\	
	&\leq\nor{e^{\sigma\Delta t}\pts{z^0-\int_{\Omega}z^0\dx}}_{L^2(\Omega)}\\
	&+\nor{\int_0^te^{\sigma\Delta (t-s)}\pts{g(y^{\sigma}(s,\cdot),z^{\sigma}(s,\cdot))-\int_{\Omega}g(y^{\sigma}(s,x),z^{\sigma}(s,x))\dx}\dd{s}}_{L^2(\Omega)}\\
&+\abs{\int_0^t\int_{\Omega}g(y^{\sigma}(s,x),z^{\sigma}(s,x))-g(y(s,x),\xi(s))\dx\dd{s}}=:M_1+M_2+M_3.
\end{aligned}
\end{equation}

{\bf Step 3:} In this step, let us estimate each term $M_i$, $1\leq i\leq 3$. Applying property $ii)$ of Lemma~\ref{A.2}, with $p=q=2$ and $x_0=z^0-\int_{\Omega}z^0\,dx$, we obtain the following estimate for $M_1$
\begin{equation}\label{e3.205}
\begin{aligned}
	%t^{\frac{n}{2}}\nor{e^{\sigma\Delta t}\pts{z^0-\int_{\Omega}%z^0\,dx}}_{L^2(\Omega)}
	t^{\frac{n}{2}}M_1&\leq Ct^{\frac{n}{2}}e^{-\lambda_1\sigma t}\nor{z^0}_{L^2(\Omega)}\\
	&\leq C\sigma^{-\frac{n}{2}}(\sigma t)^{\frac{n}{2}}e^{-\lambda_1\sigma t}\nor{z^0}_{L^2(\Omega)}\\
	&\leq C\sigma^{-\frac{n}{2}}\sup_{s\geq 0}s^{\frac{n}{2}}e^{-\lambda_1 s}\nor{z^0}_{L^2(\Omega)}\\
	&\leq C\sigma^{-\frac{n}{2}}\nor{z^0}_{L^2(\Omega)},
\end{aligned}
\end{equation}
for $t\in[0,T]$. 

To estimate $M_2$, let us define $R(s,x)=R_1(s,x)-R_2(s)$ where $R_1(s,x)=g(y^{\sigma}(s,x),z^{\sigma}(s,x))$ and $R_2(s)=\int_{\Omega}R_1(s,x)\,dx$. First, we show that $R(s,\cdot)$ is uniformly bounded in $L^2(\Omega)$ with respect $\sigma\geq 1$. Indeed, using that $g$ is Lipschitz and $g(0,0)=0$, by applying the energy estimate \eqref{energy-estimate-sigma}, we have
\begin{equation}\label{e3.209}
\begin{aligned}
	\nor{R_1(s,\cdot)}_{L^2(\Omega)}^2&\leq C\pts{\nor{y^{\sigma}(s,\cdot)}_{L^2(\Omega)}^2 +\nor{z^{\sigma}(s,\cdot)}_{L^2(\Omega)}^2}\\
	&\leq C\pts{\nor{y^0}_{L^2(Q_T)}^2+\nor{z^0}_{L^2(Q_T)}^2}, 
\end{aligned} 
\end{equation}
where $C$ is independent of $\sigma\geq 1$. To conclude that $R(s,\cdot)$ is uniformly bounded in $L^2(\Omega)$ with respect to $\sigma\geq 1$ we apply Jensen inequality and \eqref{e3.209} to get
\begin{equation*}%\label{e3.210}
\begin{aligned}
	\nor{R_2(s,\cdot)}_{L^2(\Omega)}^2&\leq \int_{\Omega}\abs{\int_{\Omega}R_1(s,\bar{x})\dd\bar{x}}^2\dx\leq C \int_{\Omega}\abs{R_1(s,x)}^2\dx \\
	&\leq C\pts{\nor{y^0}_{L^2(Q)}^2+\nor{z^0}_{L^2(Q)}^2}.
\end{aligned} 
\end{equation*}

Rewriting $M_2$ and by H\"older inequality, we have
\begin{equation*}
	M_2=\nor{\int_0^te^{\sigma\Delta (t-s)}R(s,\cdot)\dd{s}}_{L^2(\Omega)}\leq\int_0^t\nor{e^{\sigma\Delta (t-s)}R(s,\cdot)}_{L^2(\Omega)}\dd{s}.
\end{equation*}
We note that $R\in L^2(\Omega)$ and by definition of $R$ we have that $\int_{\Omega}R(s,x)\dx=0$. Therefore we can use the property $ii)$ of the Lemma~\ref{A.2} with $p=q=2$ and $x_0=R(s,x)$ to obtain 
\begin{equation}\label{e3.213}
\begin{aligned}
	\sup\limits_{0\leq t\leq T}t^{\frac{n}{2}}M_2&\leq\sup\limits_{0\leq t\leq T}t^{\frac{n}{2}}\int_0^t\nor{e^{\sigma\Delta (t-s)}R(s,\cdot)}_{L^2(\Omega)}\dd{s}	\\
	&\leq\sup\limits_{0\leq t\leq T} t^{\frac{n}{2}}C\int_0^te^{-(t-s)\lambda_1\sigma}\dd{s}\nor{R(s,\cdot)}_{L^2(\Omega)}\\
	&\leq C \sup\limits_{0\leq t\leq T}t^{\frac{n}{2}}\frac{1-e^{-t\lambda_1\sigma}}{\lambda_1\sigma}\sup\limits_{0\leq t\leq T}\nor{R(s,\cdot)}_{L^2(\Omega)}\\
	&\leq\frac{C}{\lambda_1\sigma}\sup\limits_{0\leq t\leq T}\nor{R(s,\cdot)}_{L^2(\Omega)}.
\end{aligned}
\end{equation}

To finish this step, we estimate $M_3$ as follows. Using that $g$ is a Lipschitz function and applying the H\"older inequality we get
\begin{equation}\label{e3.214}
\begin{aligned}
	&M_3=\abs{\int_0^t\int_{\Omega}g(y^{\sigma}(s,x),z^{\sigma}(s,x))-g(y(s,x),\xi(s,x))\dx\dd{s}}\\
	&\leq C_g\int_0^t\int_{\Omega}\abs{(y^{\sigma}(s,x)-y(s,x)}+\abs{z^{\sigma}(s,x))-\xi(s,x)}\dx\dd{s}\\
	&\leq C\int_0^t\pts{\nor{y^{\sigma}-y}_{L^2(\Omega)} +\nor{z^{\sigma}-\xi}_{L^2(\Omega)}}\dd{s}. 
\end{aligned}
\end{equation}

{\bf Step 4:} In this step we use the inequality \eqref{e3.204} together with the estimates of $M_1$, $M_2$ and $M_3$ to conclude the proof. Putting together \eqref{e3.204} and \eqref{e3.214} allows us to write
\begin{equation}\label{e3.215}
\begin{aligned}
	\nor{z^{\sigma}(t,\cdot)-\xi(t)}_{L^2(\Omega)}\leq A(t)+C\int_0^tY(s)\dd{s},
\end{aligned}
\end{equation}
where 
\begin{equation*}\label{e3.216}
	Y(t)=\nor{y^{\sigma}(s,x)-y(s,\cdot)}_{L^2(\Omega)}+\nor{z^{\sigma}(t,\cdot)-\xi(t)}_{L^2(\Omega)},
\end{equation*}
\begin{equation*}\label{e3.217}
    A(t)=M_1(t)+M_2(t),
\end{equation*}
for $t\in(0,T)$. Moreover, for any $t_0\in (0,T)$, we can choose $\varepsilon\in (0,t_0)$ and using inequality \eqref{e3.215} we obtain the following estimate 
\begin{equation*}\label{e3.219}
\begin{aligned}
	&\nor{z^{\sigma}(t,\cdot)-\xi(t)}_{L^2(\Omega)}\leq A_{\sigma,\varepsilon,T}+C\varepsilon^{\frac{1}{2}} \pts{\int_0^\varepsilon\abs{Y(s)}^2\,ds}^{\frac{1}{2}}\\
    &\hspace{0.5cm}+C\int_{\varepsilon}^t\nor{y^{\sigma}(s,x)-y(s,\cdot)}_{L^2(\Omega)}\dd{s}
	+C\int_{\varepsilon}^t\nor{z^{\sigma}(t,\cdot)-\xi(t)}_{L^2(\Omega)}\dd{s},
\end{aligned}
\end{equation*}
for $t\in(\varepsilon,T)$, where $A_{\sigma,\varepsilon,T}=\sup\limits_{t\in\cts{\varepsilon,T}}A(t)$. Applying Gronwall's inequality
\begin{equation*}%\label{e3.220}
\begin{aligned}
	&\nor{z^{\sigma}(t,\cdot)-\xi(t)}_{L^2(\Omega)}\\
	& \quad \leq   \left(A_{\sigma,\varepsilon,T}+C\varepsilon^{\frac{1}{2}} \pts{\int_0^\varepsilon\abs{Y(s)}^2\,ds}^{\frac{1}{2}} \right. \\
    &\left. \qquad\quad +\, C\int_{\varepsilon}^T\nor{y^{\sigma}(s,x)-y(s,\cdot)}_{L^2(\Omega)}\,ds\right)e^{C(T-\varepsilon)},
\end{aligned}
\end{equation*}
for $t\in(\varepsilon,T)$. 

Notice that estimates \eqref{e3.205} and \eqref{e3.213} imply that $A_{\sigma,\varepsilon,T}$ tends towards $0$ when $\sigma\rightarrow \infty$. On the other hand, using the definition of $Y$ and applying triangle inequality give
\begin{equation}\label{e3.216a}
\small
\begin{aligned}
	&\int_0^\varepsilon\abs{Y(s)}^2\dd{s}\leq C\pts{\nor{y^{\sigma}-y}_{L^2(0,\varepsilon;L^2(\Omega))}^2+\nor{z^{\sigma}-\xi}_{L^2(0,\varepsilon;L^2(\Omega))}^2}\\
    &\hspace{0.1cm}\leq C\pts{\nor{y^{\sigma}}_{L^2(0,T;L^2(\Omega))}^2+\nor{y}_{L^2(0,T;L^2(\Omega))}^2+\nor{z^{\sigma}}_{L^2(0,T;L^2(\Omega))}^2+\nor{\xi}_{L^2(0,T)}^2}.
\end{aligned}
\end{equation}
Recall from \eqref{strong-converge} that we have $z^{\sigma}\rightarrow \xi$ in $L^2(0,T)$, then $\xi$ satisfies \eqref{e3.195}, that is to say
\begin{equation}\label{estimate-xi}
	\nor{\xi}_{L^2(0,T)}^2\leq C\pts{\nor{y^0}_{L^2(\Omega)}^2+\nor{z^0}_{L^2(\Omega)}^2},
\end{equation} 
where $C$ is independent of the diffusion coefficient $\sigma\geq 1$. 

Using \eqref{estimate-xi} together with \eqref{energy-estimate-sigma} and \eqref{e3.195} allows us to estimate the right-hand side of \eqref{e3.216a} as follows 
\begin{equation*}
	\int_0^\varepsilon\abs{Y(s)}^2\dd{s}\leq C\nor{(y^0,z^0)}_{L^2(\Omega)}^2.
\end{equation*}
Therefore, the term $\pts{\int_0^\varepsilon\abs{Y(s)}^2\dd{s}}^{1/2}$ is uniformly  bounded with respect to $\sigma$ and $\varepsilon$.

From Step~1, we have that $y^{\sigma}\rightarrow y$ in $L^2(Q)$, then the term $\int_{\varepsilon}^T\nor{y^{\sigma}(s,x)-y(s,\cdot)}_{L^2(\Omega)}\dd{s}$ tend to zero when $\sigma\rightarrow\infty$. Then for every $t_0\in (0,T)$ and one can choose $\varepsilon\in (0,t_0)$ small enough so 
\begin{equation}\label{e3.223}
	\lim_{\sigma\to \infty}\sup_{t\in[t_0,T]}\nor{z^{\sigma}(t,\cdot)-\xi(t)}_{L^2(\Omega)}=0.
\end{equation}
Moreover, using the limit \eqref{e3.223} and the fact that $z^{\sigma}(T)=0$ for every $\sigma\geq 1$, we have that $\xi(T)=0$. To conclude, we note that by the uniqueness of the limit we obtain that $z=\xi$ in $(t_0,T)\times\Omega$. This ends the proof of \Cref{thm:main_1}.

\bigskip

\bmhead{Acknowledgments}
This work has received support from Project A1-S-17475 of CONACyT, Mexico, and by Projects IN109522 and IN104922 of DGAPA-UNAM, Mexico. The work of the first author was supported by the program ``Estancias Posdoctorales por México para la Formación y Consolidación de las y los Investigadores por México'' of CONACyT while the second author was supported by the program ``Becas Nacionales'' of the same institution. 

The first author would like to thank all members of the Departments of Mathematics and Mechanics of IIMAS-UNAM for their kind hospitality during his research stay which was useful for developing the first version of this manuscript. He would also like to thank Prof. Luz de Teresa (IM-UNAM) and Prof. Kévin Le Balc'h (INRIA) for fruitful discussions about the controllability of coupled parabolic systems.

\begin{appendices}

\section{Appendix}\label{Apendice}
\begin{proposition}\label{A.1}
For any $\epsilon>0$ there exists $\lambda_0>0$ such that, for every $\lambda>\lambda_0$ such that
\begin{equation}\label{e1.40}
	e^{\hat\alpha s}\leq e^{\pts{1+\epsilon}\alpha^*s},
\end{equation}
for all $s>0$.
\end{proposition}
The proof of this Lemma is similar to \cite{LM18}. The following results appears in \cite{MCHSKS18}.
\begin{lem}\label{A.2}
Let $ \lvs{e^{t\sigma\Delta}}$ be the heat semigroup with the Neumann boundary conditions in a bounded domain $\Omega\subset \mathbb R^n$ with smooth boundary and such that $\abs{\Omega}=1$.
\begin{enumerate}{\roman{enumi}}
\item For every $C\in\mathbb R$, we have $e^{t\sigma\Delta}C=C$ for all $t\geq 0$.
\item $\nor{e^{t\sigma\Delta}z_0}_{L^2(\Omega)}\leq C\pts{1+\pts{t\sigma}^{-\frac{n}{2}\pts{\frac{1}{q}-\frac{1}{p}}}}e^{-\lambda_1\sigma t}\nor{z_0}_{L^2(\Omega)}$.
\end{enumerate}  
\end{lem}
\end{appendices}

\bibliographystyle{plain}
\small{\bibliography{sn-bibliography}}

\end{document}